\documentclass[11pt,reqno]{amsart}
\usepackage[margin=1in]{geometry}

\usepackage{amssymb}
\usepackage{graphicx} 
\usepackage{mathrsfs}
\usepackage{enumerate}
\usepackage{xspace}
\usepackage{comment}
\usepackage{color}
\usepackage{enumerate}
\usepackage[shortlabels]{enumitem}
\usepackage{amsthm}
\usepackage{mathtools}
\mathtoolsset{showonlyrefs}
\usepackage{dsfont}
\usepackage{bbm}
\usepackage[colorlinks=true,linkcolor=blue]{hyperref}
\usepackage{mathscinet}
\usepackage[backend=biber,sorting=nty,natbib=true,isbn=false,doi=false,url=false,maxbibnames=5,style=ieee]{biblatex}
\usepackage[backgroundcolor=white,bordercolor=red]{todonotes}
\setuptodonotes{inline}
\DeclareMathAlphabet{\mathpzc}{OT1}{pzc}{m}{it}
\numberwithin{equation}{section}

\DeclareSourcemap{
  \maps[datatype=bibtex]{
    \map{
      \step[fieldset=language, null]
    }
  }
}
\makeatletter
\def\@settitle{\begin{center}%
		\baselineskip14\p@\relax
		\bfseries
		\uppercasenonmath\@title
		\@title
		\ifx\@subtitle\@empty\else
		\\[1ex]\uppercasenonmath\@subtitle
		\footnotesize\mdseries\@subtitle
		\fi
	\end{center}%
}
\def\subtitle#1{\gdef\@subtitle{#1}}
\def\@subtitle{}

\makeatletter
\def\@settitle{\begin{center}%
		\baselineskip14\p@\relax
		\bfseries
		\uppercasenonmath\@title
		\@title
		\ifx\@subtitle\@empty\else
		\\[1ex]\uppercasenonmath\@subtitle
		\footnotesize\mdseries\@subtitle
		\fi
	\end{center}%
}
\def\subtitle#1{\gdef\@subtitle{#1}}
\def\@subtitle{}
\makeatother

\usepackage{bbm}
\usepackage[normalem]{ulem}
\usepackage{cancel}

\newcommand{\cA}{\mathcal{A}}

\newcommand{\cE}{\mathcal{E}}
\newcommand{\cF}{\mathcal{F}}
\newcommand{\cG}{\mathcal{G}}

\newcommand{\cK}{\mathcal{K}}

\newcommand{\cR}{\mathcal{R}}

\newcommand{\cU}{\mathcal{U}}

\newcommand{\fP}{\mathfrak{P}}

\newcommand{\bN}{\mathbb{N}}

\newcommand{\bR}{\mathbb{R}}

\begin{document}




\theoremstyle{plain}

\newtheorem{theorem}{Theorem}[section]
\newtheorem{lemma}[theorem]{Lemma}
\newtheorem{proposition}[theorem]{Proposition}
\newtheorem{corollary}[theorem]{Corollary}
\newtheorem{assumption}[theorem]{Assumption}
\newtheorem{condition}[theorem]{Condition}
\newtheorem{definition}[theorem]{Definition}
\newtheorem*{lemmaOhne}{Lemma}

\theoremstyle{definition}

\newtheorem{example}[theorem]{Example}
\newtheorem{remark}[theorem]{Remark}
\newtheorem{SA}[theorem]{Assumption}
\newtheorem{discussion}[theorem]{Discussion}
\newtheorem{remarks}[theorem]{Remark}
\newtheorem*{notation}{Remark on Notation}
\newtheorem{application}[theorem]{Application}

\newcommand{\of}{[\hspace{-0.06cm}[}
\newcommand{\gs}{]\hspace{-0.06cm}]}

\newcommand\llambda{{\mathchoice
		{\lambda\mkern-4.5mu{\raisebox{.4ex}{\scriptsize$\backslash$}}}
		{\lambda\mkern-4.83mu{\raisebox{.4ex}{\scriptsize$\backslash$}}}
		{\lambda\mkern-4.5mu{\raisebox{.2ex}{\footnotesize$\scriptscriptstyle\backslash$}}}
		{\lambda\mkern-5.0mu{\raisebox{.2ex}{\tiny$\scriptscriptstyle\backslash$}}}}}

\newcommand{\1}{\mathds{1}}

\newcommand{\F}{\mathbf{F}}
\newcommand{\G}{\mathbf{G}}

\newcommand{\B}{\mathbf{B}}

\newcommand{\M}{\mathcal{M}}

\newcommand{\la}{\langle}
\newcommand{\ra}{\rangle}

\newcommand{\lle}{\langle\hspace{-0.085cm}\langle}
\newcommand{\rre}{\rangle\hspace{-0.085cm}\rangle}
\newcommand{\blle}{\Big\langle\hspace{-0.155cm}\Big\langle}
\newcommand{\brre}{\Big\rangle\hspace{-0.155cm}\Big\rangle}

\newcommand{\X}{X}

\newcommand{\tr}{\operatorname{tr}}
\newcommand{\N}{{\mathbb{N}}}
\newcommand{\cadlag}{c\`adl\`ag }
\newcommand{\on}{\operatorname}
\newcommand{\oP}{\overline{P}}
\newcommand{\oO}{\mathcal{O}}
\newcommand{\D}{\mathsf{D}} 
\newcommand{\bx}{\mathsf{x}}
\newcommand{\bb}{\hat{b}}
\newcommand{\bs}{\hat{\sigma}}
\newcommand{\bv}{\hat{v}}
\renewcommand{\v}{\mathfrak{m}}
\newcommand{\ob}{\bar{b}}
\newcommand{\oa}{\bar{a}}
\newcommand{\os}{\widehat{\sigma}}
\renewcommand{\j}{\varkappa}
\newcommand{\scl}{\ell}
\newcommand{\Y}{\mathscr{Y}}
\newcommand{\Z}{\mathscr{Z}}
\newcommand{\T}{\mathcal{T}}
\newcommand{\con}{\mathsf{c}}
\newcommand{\nk}{\hspace{-0.25cm}{{\phantom A}_k^n}}
\newcommand{\nl}{\hspace{-0.25cm}{{\phantom A}_1^n}}
\newcommand{\nm}{\hspace{-0.25cm}{{\phantom A}_2^n}}
\newcommand{\n}{\hspace{-0.35cm}{\phantom {Y_s}}^n}
\newcommand{\nme}{\hspace{-0.35cm}{\phantom {Y_s}}^{n - 1}}
\renewcommand{\o}{\hspace{-0.35cm}\phantom {Y_s}^0}
\newcommand{\e}{\hspace{-0.4cm}\phantom {U_s}^1}
\newcommand{\z}{\hspace{-0.4cm}\phantom {U_s}^2}
\newcommand{\iii}{|\hspace{-0.05cm}|\hspace{-0.05cm}|}
\newcommand{\co}{\overline{\on{co}}}
\renewcommand{\k}{\mathsf{k}}
\newcommand{\ovb}{\overline{b}}
\newcommand{\ova}{\overline{a}}
\newcommand{\s}{\mathfrak{s}}
\newcommand{\opsi}{\overline{\Psi}}
\newcommand{\ol}{\mathcal{L}}
\newcommand{\oD}{\overline{D}}
\newcommand{\ua}{\underline{a}}
\newcommand{\ou}{\overline{b}}
\newcommand{\uu}{\underline{b}}

\renewcommand{\epsilon}{\varepsilon}
\renewcommand{\rho}{\varrho}

\newcommand{\fPs}{\fP_{\textup{sem}}}
\newcommand{\fPas}{\mathfrak{S}_{\textup{ac}}}
\newcommand{\rrarrow}{\twoheadrightarrow}
\newcommand{\ocA}{\mathcal{U}}
\newcommand{\bth}{\overset{\leftarrow}\theta}
\renewcommand{\th}{\theta}
\newcommand{\fPasn}{\mathfrak{S}^{\textup{ac}, n}_{\textup{sem}}}
\newcommand{\CLM}{\mathfrak{M}^\textup{ac}_\textup{loc}}
\newcommand{\Sd}{\mathcal{S}^\textup{sp}_{\textup{d}}}
\newcommand{\Sc}{\mathcal{S}}
\newcommand{\Sac}{\mathcal{S}_\textup{ac}}
\newcommand{\A}{\mathsf{A}}
\newcommand{\Td}{\mathsf{T}^\textup{d}}
\renewcommand{\t}{\mathfrak{t}}
\newcommand{\UC}{\hspace{-0.03cm}\textit{UC}}

\newcommand{\nnabla}{\nabla}
\newcommand{\f}{\mathfrak{f}}
\newcommand{\g}{\mathfrak{g}}
\newcommand{\oconv}{\overline{\on{co}}\hspace{0.075cm}}
\renewcommand{\a}{\mathfrak{a}}
\renewcommand{\b}{\mathfrak{b}}
\renewcommand{\d}{\mathsf{d}}
\newcommand{\bS}{\mathbb{S}^d_+}
\newcommand{\p}{\mathsf{p}}
\newcommand{\dr}{r} 
\newcommand{\m}{\mathbb{M}} 
\newcommand{\Q}{Q}
\newcommand{\usc}{\textit{USC}}
\newcommand{\lsc}{\textit{LSC}}
\renewcommand{\q}{\mathfrak{q}}
\newcommand{\W}{\mathscr{W}}
\newcommand{\w}{\mathsf{w}}
\newcommand{\oM}{\mathsf{M}}
\newcommand{\oZ}{\mathsf{Z}}
\newcommand{\oK}{\mathsf{K}}
\renewcommand{\Re}{\operatorname{Re}}
\newcommand{\cCk}{\mathsf{c}_k}
\newcommand{\C}{\mathsf{C}}
\newcommand{\oPi}{\overline{\Pi}}
\renewcommand{\P}{\mathbf{P}}
\renewcommand{\Q}{\mathsf{Q}}
\renewcommand{\Y}{\mathsf{Y}}
\renewcommand{\Z}{\mathsf{Z}}
\newcommand{\E}{\mathbf{E}}
\renewcommand{\Q}{\mathbf{Q}}
\renewcommand{\bS}{\mathbb{S}}
\newcommand{\USA}{\textit{USA}}
\renewcommand{\p}{\partial_t}
\renewcommand{\d}{d}
\newcommand{\ca}{\on{ca} \hspace{0.01cm}}
\renewcommand{\cR}{\mathcal{R}}
\newcommand{\RPF}{\mathsf{RPF}}
\renewcommand{\c}{\alpha\hspace{0.02cm}} 
\newcommand{\cFs}{\cF_s}
\newcommand{\cFt}{\cF_t}
\newcommand{\hcE}{\widehat{\cE}}
\newcommand{\hcU}{\widehat{\cU}}
\newcommand{\hc}{\widehat{\c}}
\newcommand{\hT}{\widehat{T}}
\newcommand{\lowersup}[1]{^{\raisebox{-0.3ex}{$\scriptstyle #1$}}}

\newcommand{\Lip}{\textit{Lip}}
\newcommand{\urange}{\bR}

\renewcommand{\emptyset}{\varnothing}

\newcommand{\ff}[1]{{\color{green!60!black}#1}}
\newcommand{\om}{\omega}
\newcommand{\tom}{\widetilde{\omega}}
\newcommand{\tlam}{\widetilde{\lambda}}

\allowdisplaybreaks

\makeatletter

\title[Risk-sensitive exit-time control for SDEs]{Risk-sensitive exit-time control for stochastic differential equations with path-dependent coefficients} 
\author[D. Criens]{David Criens}
\address{D.\ Criens - University of Freiburg, Ernst-Zermelo-Str.\ 1, 79104 Freiburg, Germany.}
\email{david.criens@stochastik.uni-freiburg.de}

\author[F. Fuchs]{Fabian Fuchs}
\address{F.\ Fuchs - Department of AI, Data and Decision Sciences, Luiss Guido Carli, Viale Romania 32, 00197 Roma, Italy.}
\email{ffuchs@luiss.it}

\thanks{FF is supported by the Italian Ministry of University and Research (MIUR) within the framework of PRIN project 20223PNJ8K}
\date{\today}

\begin{abstract}
In this work, we study small-noise asymptotics of risk-sensitive exit-time control problems governed by stochastic differential equations with path-dependent coefficients. Our main result establishes the convergence of the $\log$-transformed exit-time problem to a deterministic control problem with path-dependent coefficients. For its proof, we first derive a novel variational representation for general $\log$-transformed stochastic control problems with path-dependent coefficients, combining tools from the theory of path-dependent partial differential equations and convex expectations on path spaces. In a second step, we use probabilistic methods to analyze the convergence of the resulting variational formulas.\ To illustrate the scope of our analysis, we consider a computable example for a stochastic differential equation with memory and characterize the limiting problem and associated control strategies.

   
\end{abstract}

\maketitle

\section{Introduction}

Consider the $d$-dimensional control system
\[
    dX_t = \mu(t,X,\lambda_t)\,dt+\sqrt{\varepsilon}\,
    \sigma(t,X,\lambda_t)\,dW_t,
\]
where the coefficients $\mu$ and $\sigma$ may depend on the whole past of the path $X$ in a non-anticipating way. 
The controller acts through the coefficients on the diffusion with the aim of keeping the process inside a bounded domain $D\subseteq\mathbb{R}^d$ until a goal has been reached. We formulate this as a risk-sensitive exit-time control problem, where the controller maximizes
\[
    \mathbb{E}\Bigl[\exp\bigl(g(\tau_D)/\varepsilon\bigr)\Bigr],
    \qquad
    \tau_D:=\inf\{t\geq 0 \colon X_t\notin D\},
\]
for an increasing, bounded, and continuous function \(g\).
In this form, the criterion assigns exponentially larger weights to paths with large exit times. The function $g$ determines how strongly later exits are rewarded and encodes whether the controller's goal has been reached. 
We are then interested in quantifying how small stochastic perturbations influence the controlled system and, in particular, in the limit of the logarithmically transformed value as the noise vanishes, i.e., the limit of
\[
    V^\varepsilon
    := \varepsilon \log \sup_\lambda
    \mathbb{E}\Bigl[\exp\big(g(\tau_D)/\varepsilon\big)\Bigr]
\]
as $\varepsilon\searrow 0$.
Our main result, Theorem~\ref{theo: main result}, identifies the limit of $V^\varepsilon$ as the value of a deterministic path-dependent optimal control problem, which features an additional control term in the drift of the controlled system and a quadratic penalization in the reward.

Risk-sensitive exit-time control problems, also referred to as risk-sensitive escape control problems, have been studied, for instance, in \cite{BD_01,DE_97,FlemSon_06} for diffusion models with state-dependent coefficients and uncontrolled volatility. They arise, for example, in tracking problems and safety analyses of engineered control systems; see the recent paper \cite{GoBuWi26} as well as the references therein. We emphasize that most previous works in this area focus on minimization criteria for the controller, whereas we consider a maximization problem. This distinction affects the form of the limiting object as, in our setting, the small-noise limit is a control problem rather than a two-player game. 

The main difficulty in proving any such convergence result is that the small-noise limit and optimization over the controls have to be taken simultaneously. Since limits and suprema need not commute, it is not immediate that the limiting value is obtained by replacing the stochastic dynamics by their deterministic counterpart. The literature predominantly approaches this problem in two ways: 

The PDE approach, developed by Fleming, cf.\ \cite{J_92, FlemSon_06}, relates the logarithmically transformed value $V^\varepsilon$ to a Hamilton--Jacobi--Bellman equation and identifies the limit via a comparison principle for viscosity solutions. 
The variational approach of Bou\'e and Dupuis, see \cite{BD_98,BD_01}, instead relates $V^\varepsilon$ to an auxiliary stochastic control problem with quadratic control costs. The resulting variational formula can then be used to analyze the limiting object by probabilistic methods.

The approach in this work extends and combines the PDE and variational approaches in a path-dependent control framework. 
In Theorem~\ref{theo: VF}, we build on the recent results on path-dependent PDEs (PPDEs) in \cite{Z_23} and a comparison principle for convex expectations in \cite{CK_25} to derive a variational representation for the log-transform of general
control problems with path-dependent coefficients.
We believe this result to be of independent interest to the community.
In a second step, we use probabilistic methods to identify the small-noise limit of the resulting variational formulas. This provides a path-dependent counterpart of the Bou\'e--Dupuis approach that is flexible enough to treat measurable functions of the underlying state beyond our particular choice $g(\tau_D)$.

Compared to most of the existing literature, the coefficients in our setup are allowed to depend on the past of the path; see, e.g., \cite{BVLaTa20} who deal with Sanov and Schilder problems and conjecture that a PPDE connection can be made in certain non-Markovian settings.
In the uncontrolled case with path-dependent coefficients, \cite{MRTZ_16} lifts the PDE approach to a non-Markovian setting, replacing some classical PDE arguments by tools from BSDE theory and relating the limiting object to a first-order path-dependent PDE. We also refer to \cite{PeSoWa23} for an overview of path-dependent PDE theory and to \cite{ZhToZh25, ReToZh20} and the references therein for further recent developments.

At this point, we want to highlight that the reward of the controller is discontinuous as the hitting time $\tau_D$ is discontinuous. This presents an intrinsic obstacle in the analysis of exit-time problems as the available PPDE methods seem to apply only to situations with regular boundary conditions.
Consequently, such viscosity methods seem, at the time of writing, not to be applicable to our problem, motivating our combined analytic-probabilistic approach.

The rest of the paper is organized as follows: In Section~\ref{sec: laplace-principle}, we introduce the stochastic control framework, state the assumptions, our main result, Theorem \ref{theo: main result}, and the explicitly computable example in the case of a controlled SDE with memory.
Section~\ref{sec: variational-formula} is devoted to deriving the variational formula for the entropic transform via PPDE methods. Section \ref{sec: convergence relaxed} contains the convergence result for relaxed control rules needed to identify the small-noise limit. Finally, Section \ref{sec: proof main} uses the results of the previous sections to show the main result.

\section{Small-noise Asymptotics for Exit Problems} \label{sec: laplace-principle}

We start by introducing the mathematical framework for the control system
\begin{align} \label{eq: control system main section}
    dX_t = \mu(t,X,\lambda_t)\,dt + \sqrt{\varepsilon} \, \sigma(t,X,\lambda_t)\,dW_t
\end{align}
that we investigate in this paper. 
For $d \in \bN$, let \(\Omega := C (\bR_+; \bR^d)\) be the set of continuous $\bR^d$-valued paths, endowed with the local uniform topology, and denote the coordinate map \(\X_t (\omega) = \omega (t)\) for all \(t \in \bR_+\) and \(\omega \in \Omega\). Furthermore, let \(\cF := \mathcal{B}(\Omega) = \sigma (\X_s, s \in \bR_+)\) and \(\cF_t := \sigma (\X_s, s \in [0, t])\) for \(t \in \bR_+\). 
Moreover, we suppose that the action space \(\Lambda\) is a compact, metrizable space and consider two coefficients
\begin{align*}
	&\mu \colon \bR_+ \times \Omega \times \Lambda \to \bR^d, \qquad 
	\sigma \colon \bR_+ \times \Omega \times \Lambda \to \bR^{d \times r},
\end{align*}
where \(r \in \mathbb{N}\) represents the dimension of the noise. Throughout we impose the following conditions. For transparency, in our proof sections, we indicate precisely which parts of the assumption are used. 

\begin{assumption} \label{SA: control}
Throughout, we impose the following assumptions:
\begin{enumerate}[(A)]
\item	The coefficients \(\mu\) and \(\sigma\) are Borel measurable and, for every $\omega \in \Omega$, the map \((t,\lambda) \mapsto (\mu (t, \omega, \lambda), \sigma (t, \omega, \lambda))\) is continuous. \label{item:SA:continuity in control}

\item For every \(\lambda \in \Lambda\), the maps \((t, \omega) \mapsto \mu (t, \omega, \lambda)\) and \((t, \omega) \mapsto \sigma (t, \omega, \lambda)\) are non-anticipative,
	i.e.,  \(\mu (t, \omega, \lambda)\) and \(\sigma (t, \omega, \lambda)\) depend on \(\omega\) only through \((\omega (s))_{s \leq t}\). 
\item\label{item:SA:Lipschitz} The coefficients \(\mu\) and \(\sigma\) are uniformly equi-Lipschitz, i.e., for every \(T > 0\), there exists a constant \(L = L_T > 0\) such that, for all \(t \in [0, T], \omega, \omega' \in \Omega\) and \(\lambda \in \Lambda\), 
\[
\| \mu (t, \omega, \lambda) - \mu (t, \omega', \lambda) \| + \| \sigma (t, \omega, \lambda) - \sigma (t, \omega', \lambda) \| \leq L \| \omega - \omega' \|_t, 
\]
where \(\| \omega \|_t := \sup_{s \in [0, t]} \| \omega (s) \|\).
\item\label{item:SA:growth} The coefficients \(\mu\) and \(\sigma\) are of linear growth, i.e., for every \(T > 0\), there exists a constant \(C = C_T > 0\) such that, for all \(t \in [0, T], \omega \in \Omega\) and \(\lambda \in \Lambda\), 
\[
\| \mu (t, \omega, \lambda)\| + \| \sigma (t, \omega, \lambda) \| \leq C ( 1 + \| \omega \|_t ).
\]

\item\label{item:SAg} The reward \(g \colon [0, \infty] \to \bR\) is an increasing, bounded, and continuous function. 
\end{enumerate} 
\end{assumption}

We now introduce the relaxed control framework that we use to formalize the control system \eqref{eq: control system main section}. For comments on the relation to other control frameworks, see Remark~\ref{rem: other control frameworks} below. 
In the relaxed control setting, controls are modeled as elements of the space \(\m\) of all Radon measures on \(\bR_+ \times \Lambda\), whose projections to \(\bR_+\) coincide with the Lebesgue measure. 
We endow \(\m\) with the vague topology, which turns it into a compact metrizable space; see \cite[Theorem I.2.2]{EKNJ88}. 
The coordinate map on \(\m\) is denoted by \(M\).
Define the \(\sigma\)-field \[\M := \sigma (M_{t, s} (\phi)\colon \, t < s,\, \phi \in C_{c} (\bR_+ \times \Lambda; \bR)),\] where
\[
M_{t, s} (\phi) := \int_t^s \int_\Lambda \phi (r, \lambda) \, M(dr, d\lambda).
\]
The product space \((\Omega \times \m, \cF \otimes \M)\) captures the controlled process together with its control. Adapting the above notation, we denote the coordinate process on this space by \((X, M)\). 

For \(\varepsilon > 0\) and \(\varphi \in C^2_b (\bR^d; \bR)\), we define
\begin{align} \label{eq: C}
	N_{t, s}^\varepsilon (\varphi) := \varphi (X_s) - \int_{t}^{s} \int_\Lambda  L^\varepsilon (X_r, r, X, \lambda, \varphi ) \, M (dr , d \lambda), 
\end{align}
where
\begin{align} \label{eq: L}
	L^\varepsilon (x, s, \omega, \lambda, \varphi) := \langle \mu (s, \omega, \lambda), \nabla \varphi (x) \rangle + \tfrac{\varepsilon}{2} \on{tr} \big[ \sigma (s, \omega, \lambda) \sigma^* (s, \omega, \lambda) \, \nabla^2 \varphi (x) \big].
\end{align}
A relaxed control rule with initial value \((t, \omega) \in \bR_+ \times \Omega\) is then a probability measure \(P\) on the product space \((\Omega \times \m, \mathcal{F} \otimes \M)\) with
\[P (X = \omega \text{ on } [0 , t]) = 1,\] 
such that, for every  \(\varphi \in C^2_b (\bR^d; \bR)\), the process \((N^\varepsilon_{t, s} (\varphi))_{s \geq t}\) is a \(P\)-martingale for the filtration
\[
\cG_r := \sigma \big(X_s, M_{t, s} (\phi)\colon\, t \leq s \leq r,\, \phi \in C_c (\bR_+ \times \Lambda; \bR)\big), \quad r \geq t.
\]
For \((t, \omega) \in \bR_+ \times \Omega\), we define 
\begin{align} \label{eq: KR}
	\cK^\varepsilon (t, \omega)&:= \Big\{ \text{all relaxed control rules with initial value } (t, \omega) \Big \}.
\end{align}
The control problem associated to these control rules is given by the sublinear expectation 
\[
\cE_t^\varepsilon (\varphi) (\omega) := \sup_{ P \in \cK^\varepsilon (t, \omega)} E^P \big[ \varphi \big], 
\]
where \(\varphi \colon \Omega \to [- \infty, \infty]\) is an upper semianalytic function, i.e., for every \(t \in \bR\), the upper level set \(\{ \varphi \geq t \}\) is analytic in \(\Omega\). We stress that any Borel function is clearly upper semianalytic. Furthermore, any upper semianalytic function is universally measurable, cf.\  \cite[Section~7.7]{bershre}.

From now on, we fix a bounded domain \(D \subseteq \bR^d\) and set 
\[
\tau_D^t = \inf \{s \geq t \colon \X_s \not \in D \}.
\]
For a reward function \(g \colon [0, \infty] \to \bR\) as in Assumption~\ref{SA: control}~\ref{item:SAg}, 
we are interested in the limit of 
\[
V^\varepsilon_{t, \omega} := \varepsilon \log \mathcal{E}_t \big(e^{g (\tau_D^t) / \varepsilon} \big) (\omega)
\]
as \(\varepsilon \searrow 0\). 
\begin{remark} \label{rem: other control frameworks}
In the proof of \eqref{eq: to show 2} below, we show that the map \(\tau^t_D\) is lower semicontinuous on the path space~\(\Omega\). Using this observation, under mild additional assumptions on the time-regularity of the coefficients, it follows from \cite[Theorem~4.10]{ElKa15}, taking also localizations in the spirit of \cite[Remark~4.11]{ElKa15} into account, that \(V^\varepsilon_{t, \omega}\) has a strong (and weak) control formulation. As these frameworks are commonly used in applications, let us describe the precise statement of the relation to the strong framework under the simplifying assumption \(d = r\). 
Let \(\mathbb{W}\) be the \(d\)-dimensional Wiener measure on \((\Omega, \cF)\) and denote by \((\overline{\cF}_t)_{t \geq 0}\) the \(\mathbb{W}\)-augmentation of the canonical filtration \((\cF_t)_{t \geq 0}\). 
Moreover, let \(\mathcal{A}\) be the set of all \(\Lambda\)-valued \((\overline{\cF}_t)_{t \geq 0}\)-progressively measurable processes. Then, for every \((t, \omega) \in \bR_+ \times \Omega\) and \((\lambda_s)_{s \geq 0} \in \mathcal{A}\), standard existence and uniqueness results, cf.\ \cite[Theorem~14.30]{jacod79}, show that the SDE 
\[
\begin{cases} d Y^{\lambda, t, \omega, \varepsilon}_s = \mu (s, Y^{\lambda, t, \omega, \varepsilon}, \lambda_s) \, ds + \sqrt{\varepsilon} \, \sigma (s, Y^{\lambda, t, \omega, \varepsilon}, \lambda_s) \, d X_s, & s > t, \\ Y^{\lambda, t, \omega, \varepsilon}_s = \omega (s), & s \leq t, \end{cases} 
\]
has a (pathwise) unique strong solution process \((Y^{\lambda, t, \omega, \varepsilon}_s)_{s \geq 0}\) on the filtered probability space \((\Omega, \cF, (\overline{\cF}_s)_{s \geq 0}, \mathbb{W})\), where we recall that \(X\) is a \(\mathbb{W}\)-Brownian motion. Assume the following uniform time-regularity assumption on the coefficients \(\mu\) and \(\sigma\): 
For every \(T > 0\) and \(N > 0\), there exists a modulus of continuity \(w = w_{T, N}\) such that
\[
\| \mu (t, \omega (\cdot \, \wedge s), \lambda) - \mu(s, \omega (\cdot \, \wedge s), \lambda) \| + \| \sigma (t, \omega (\cdot \, \wedge s), \lambda)  - \sigma (s, \omega (\cdot \, \wedge s), \lambda) \| \leq w (| t - s| )
\]
for all \(0 \leq s \leq t \leq T, \omega \in \Omega\) with \(\| \omega \|_T \leq N\), and \(\lambda \in \Lambda\).
Then, a localization of \cite[Theorem~4.10]{ElKa15} shows that 
\begin{align*}
    V^{\varepsilon}_{t, \omega} = \varepsilon \log \sup_{\lambda \in \mathcal{A}} E^\mathbb{W} \Big[ \exp \Big( \frac{g (\tau_D^t (Y^{\lambda, t, \omega, \varepsilon}))}{\varepsilon} \Big) \Big], 
\end{align*}
which is a strong control formulation.

In this paper, we choose to present the problem in its relaxed formulation as it will be necessary for the proof of the variational formula in Theorem \ref{theo: VF} and we aim to illustrate how the PDE and variational approaches can be combined in a cohesive manner.
\end{remark} 

In our main theorem, Theorem \ref{theo: main result}, below, we show that the limiting object of \(V^\varepsilon_{t, \omega}\) as \(\varepsilon \searrow 0\) is given by the following deterministic control problem 
\begin{align*}
    V^0_{t, \omega} := \sup_{\substack {h \in \mathbb{H} \\ m \in \m}} \Big\{ g (\tau_D^t (Y^{h, m, t, \omega})) - \frac{1}{2} \int_0^\infty \|\dot{h} (s)\|^2 \,ds \Big\}, 
\end{align*}
where the set \(\mathbb{H} = \{ h = \int_0^\cdot \dot{h} (s) \, ds \colon \dot{h} \in L^2 (\bR_+; \bR^r) \}\) is the Brownian Cameron--Martin space, and \(Y^{h, m, t, \omega}\) is the unique solution to the functional ODE
\begin{align*}
	\begin{cases} d Y^{h, m, t, \omega}_{s} = \int_\Lambda \big( \mu (s, Y^{h, m, t, \omega}, \lambda) + \sigma (s, Y^{h, m, t, \omega}, \lambda) \, \dot{h} (s) \big) \, m (ds, d\lambda), & \text{for } s > t, \\ Y^{h, m, t, \omega}_s = \omega (s), & \text{for } s \leq t,
	\end{cases} 
\end{align*} 
which exists due to the Lipschitz assumption on the coefficients \(\mu\) and \(\sigma\), see Assumption~\ref{SA: control}~\ref{item:SA:Lipschitz} and~\ref{item:SA:growth}, e.g., by \cite[Theorem~14.30]{jacod79}. 

Before we can state our main result, we need to introduce a final object to exclude some inconvenient boundary behavior of the limiting object. 
For \(\delta > 0\), consider the $\delta$-blow-up of $D$, \(D_\delta := \{ x \in \bR^d : \inf_{y \in D} \| x - y \| \leq \delta \}\), which is a closed set as the map \(x \mapsto \inf_{y\in D} \|x - y\|\) is continuous. 
We set 
\[
    V^{0, \delta}_{t, \omega} := \sup_{\substack {h \in \mathbb{H} \\ m \in \m}} \Big\{ g (\tau_{D_\delta}^t (Y^{h, m, t, \omega})) - \frac{1}{2} \int_0^\infty \|\dot{h} (s)\|^2 \,ds \Big\}.
\]
\begin{remark}
    In the spirit of Remark~\ref{rem: other control frameworks}, one may ask whether the value functions \(V^0\) and \(V^{0, \delta}\) can be formulated within a ``strong control formulation'' via control processes rather than relaxed controls. It is worth mentioning that such a reformulation follows as in Remark~\ref{rem: other control frameworks} for the limiting problem \(V^0\) but not for its approximation \(V^{0, \delta}\).\ The intrinsic difference between these control problems is the semicontinuity of the exit times: \(\tau_D\) is lower semicontinuous, while \(\tau_{D_\delta}\) is upper semicontinuous.
\end{remark}

Finally, define 
\[
    \T := \Big\{ (t, \omega) \in \bR_+ \times \Omega : \lim_{\delta \searrow 0} V^{0, \delta}_{t, \omega} = V^{0}_{t, \omega}\Big\}.
\]
Now, we are in the position to give the main result of this paper.
\begin{theorem} \label{theo: main result} 
Let Assumption~\ref{SA: control} hold.
For every \((t, \omega) \in \T\), every sequence \((t_n, \omega_n)_{n = 1}^\infty \subseteq \bR_+ \times \Omega\) with \((t_n, \omega_n) \to (t, \omega)\), and \(\varepsilon_n \searrow 0\), we have  
\begin{align} \label{eq: main convergence statement}
\lim_{n \to \infty} V^{\varepsilon_n}_{t_n, \omega_n} = V^0_{t, \omega}.
\end{align} 
\end{theorem}

The proof of Theorem~\ref{theo: main result} can be found in Section~\ref{sec: proof main} below. Let us briefly discuss the assumption \((t, \omega) \in \T\) from our main result.

\begin{discussion}
    In Theorem~\ref{theo: main result}, we restrict our attention to initial conditions \((t, \omega) \in \T\). This distinction is expected
    in this context; see, e.g., \cite{BD_01}. 
    In general, as \(\tau^t_{D_\delta} \searrow \tau^t_{\overline{D}}\) with \(\delta \searrow 0\), it follows from classical minimax results, see, e.g., \cite[Corollary, p.\ 407]{Ter_73}, that 
    \[
    \lim_{\delta \searrow 0} V^{0, \delta}_{t, \omega} = \overline{V}\lowersup{0}_{t, \omega} := \sup_{\substack {h \in \mathbb{H} \\ m \in \m}} \Big\{ g (\tau_{\overline{D}}^t (Y^{h, m, t, \omega})) - \frac{1}{2} \int_0^\infty \|\dot{h} (s)\|^2 \,ds \Big\}.
    \]
    In case \(\omega (t) \not \in \overline{D}\), we clearly have \(V^0_{t, \omega} = \overline{V}\lowersup{0}_{t, \omega} = g (t)\). Thus, the only interesting case is \(\omega (t) \in \overline{D}\).
    We now discuss some explicit conditions for \((t, \omega) \in \T\), restricting our attention to the Markovian case with \(\mu (s, \omega, \lambda) \equiv \mu (\omega (s), \lambda)\) and \(\sigma (s, \omega, \lambda) \equiv \sigma (\omega (s), \lambda)\). We emphasize that similar conditions can be formulated in a fully path-dependent setting. To simplify our presentation, which is intended as an exposition, we only discuss the Markovian case in detail.

    Assume the following conditions:
    \begin{enumerate}[(a)]
        \item \(D\) has a \(C^{1,1}\)-boundary.
        \item \(g\) is Lipschitz continuous with Lipschitz constant \(L_g > 0\).
        \item \label{item: domination constraint} 
        Let \(d^s (x)\) be the signed distance function to the boundary of $D$, the gradient of which is the inward normal \(n (x)\). For \(\eta > 0\), set 
        \begin{align*}
        &U_\eta := \{ x \in \overline{D} \colon 0 \leq d^s (x) < \eta \}, \quad
        B_\eta := \sup_{\substack{z\in\overline{U}_\eta\\ \lambda\in\Lambda}} \langle\mu(z,\lambda), n(z)\rangle, \quad
        \Sigma_\eta := \sup_{\substack{z\in\overline{U}_\eta\\ \lambda\in\Lambda}} \|\sigma(z,\lambda)\|.
        \end{align*}
        There exists an \(\eta > 0\) such that, for all $x \in \partial D$, $\eta$ is smaller than the injectivity radius of the exponential map of the inward normal $n(x)$,
        \[
        B_\eta < 0, \quad \Sigma_\eta > 0, \quad \frac{B^2_\eta}{2 \Sigma^2_\eta} \geq L_g, \quad \on{osc} (g) \leq \frac{2 \eta | B_\eta |}{\Sigma^2_\eta}, 
        \]
        where \(\on{osc} (g) := \sup g - \inf g\).
    \end{enumerate}
    Providing some intuition, \ref{item: domination constraint} enforces the drift to push any ``admissible'' controlled process instantaneously from the boundary of \(D\) in the direction of \(D^c\), where ``admissibility'' refers to a class of optimal controls. We emphasize that \ref{item: domination constraint} can, for example, be relaxed in the direction of distinguishing between different parts of the boundary, cf. \cite{Lions_83_I}. At this stage we only focus on an exposition of the argument rather than an optimized condition.

    Let \(\eta > 0\) as in \ref{item: domination constraint}, take \(h \in \mathbb{H}\), \(m \in \m\), and \((t, \omega) \in \bR_+ \times \Omega\) such that \(\omega (t) \in D\). We write \(B \equiv B_\eta, \Sigma \equiv \Sigma_\eta\), \(Y \equiv Y^{h, m, t, \omega}\), \(\tau_D \equiv \tau_{D}^t (Y), \tau_{\overline{D}} \equiv \tau^t_{\overline{D}} (Y)\), and set 
    \[
    \rho := \inf \{s \geq \tau_D \colon Y_s \not \in U_\eta \} - \tau_D.
    \] 
    Without loss of generality, assume that \(\tau_D<\infty\).
    By the chain rule, we obtain for every \(t \leq \rho\) that
    \begin{equation} \label{eq: useful}
    \begin{split}
        d^s(Y_{\tau_D+ t})
        &= \int_{\tau_D}^{\tau_D+ t} \int_\Lambda
            \langle \mu(Y_s,\lambda) + \sigma(Y_s,\lambda)\dot{h}(s), n(Y_s)\rangle \, m(ds,d\lambda)
        \\
        &\le B\, t + \Sigma \int_{\tau_D}^{\tau_D+ t} \|\dot{h}(s)\| \, ds .
    \end{split}
    \end{equation} 
    {\em Case 1:} We now investigate the case \(\tau_D < \tau_{\overline{D}}\). 
    Set \(t^* := (\tau_{\overline{D}} - \tau_D) \wedge \varrho\). For all \(t < t^*\), we have \(d^s (Y_{\tau_D + t}) \geq 0\) and hence,
    \[ 
        \Sigma \int_{\tau_D}^{\tau_D+ t} \|\dot{h}(s)\| \, ds \geq - B \, t.
    \]
    Letting \(t \nearrow t^*\), we obtain
    \[
        \int_{\tau_D}^{\tau_{D} + t^*} \|\dot{h}(s)\| \, ds \geq - \frac{B \, t^*}{\Sigma},
    \]
    and consequently, \(t^* < \infty\).

    {\em Case 1.1:} Refining the current case under investigation, suppose that \(\tau_{\overline{D}} - \tau_D \leq \varrho\), i.e., \(t^* = \tau_{\overline{D}} - \tau_D\). Then, we obtain that 
    \begin{align*}
        g(\tau_{\overline{D}}) &- \frac{1}{2}\int_0^{\tau_{\overline{D}}} \|\dot{h}(s)\|^2 \, ds
        \\&\le \Big( g(\tau_D) - \frac{1}{2}\int_0^{\tau_D}\|\dot{h}(s)\|^2\,ds \Big)
            + \big( g(\tau_D+t^*) - g(\tau_D) \big)
            - \frac{1}{2}\int_{\tau_D}^{\tau_D+t^*}\|\dot{h}(s)\|^2\,ds
         \\&\le \Big( g(\tau_D) - \frac{1}{2}\int_0^{\tau_D}\|\dot{h}(s)\|^2\,ds \Big)
            + \big( g(\tau_D+t^*) - g(\tau_D) \big)
            - \frac{1}{2 t^*} \Big(\int_{\tau_D}^{\tau_D+t^*}\|\dot{h}(s)\|\,ds \Big)^2
        \\
        &\le V^0_{t,\omega} + \Big( L_g - \frac{B^2}{2\Sigma^2} \Big)t^* \leq V^0_{t,\omega}.
    \end{align*}
    By hypothesis \ref{item: domination constraint}, we have \(L_g - B^2 / 2 \Sigma^2 \leq 0\) and consequently, 
    \begin{align} \label{eq: main inequality}
g(\tau_{\overline{D}}) - \frac{1}{2}\int_0^{\tau_{\overline{D}}} \|\dot{h}(s)\|^2 \, ds \leq V^0_{t,\omega}.
    \end{align} 

    {\em Case 1.2:} Suppose now that \(\varrho < \tau_{\overline{D}} - \tau_D\), i.e., \(t^* = \varrho\). Then, \(d^s (Y_{\tau_D + t^*}) = \eta\), which, together with \eqref{eq: useful}, yields
    \begin{align*}
        \int_{\tau_D}^{\tau_D + t^*} \| \dot{h} (s) \| \, ds \geq \frac{\eta + |B| t^*}{\Sigma}.
    \end{align*}
    Using the Cauchy--Schwarz inequality, this entails that 
    \[
    \int_{\tau_D}^{\tau_D + t^*} \| \dot{h} (s) \|^2 \, ds \geq \frac{(\eta + |B| t^*)^2}{\Sigma^2 t^*}.
    \]
    Minimizing \(t \mapsto (\eta + |B| t)^2 / \Sigma^2 t\) over \(t \in (0, \infty)\) yields that 
    \[
    \int_{\tau_D}^{\tau_D + t^*} \| \dot{h} (s) \|^2 \, ds \geq \frac{4 \eta |B|}{\Sigma^2}.
    \]
    We conclude that 
    \begin{align*}
        g(\tau_{\overline{D}}) - \frac{1}{2}\int_0^{\tau_{\overline{D}}} \|\dot{h}(s)\|^2 \, ds 
        &\leq V^0_{t, \omega} + \big(g (\tau_{\overline{D}}) - g (\tau_D )\big) - \frac{1}{2} \int_{\tau_D}^{\tau_{\overline{D}}} \| \dot{h}(s)\|^2 \, ds
        \\&\leq V^0_{t, \omega} + \on{osc} (g) - \frac{1}{2} \int_{\tau_D}^{\tau_{D} + t^*} \| \dot{h}(s)\|^2 \, ds
        \\&\leq V^0_{t, \omega} + \on{osc} (g) - \frac{2 \eta |B|}{\Sigma^2} \leq V^0_{t, \omega}, 
    \end{align*}
    where we used \ref{item: domination constraint} for the last inequality. 
    
    In summary, we proved that \eqref{eq: main inequality} holds whenever \(\tau_D < \tau_{\overline{D}}\). 

    \smallskip
    {\em Case 2:} If \(\tau_D = \tau_{\overline{D}}\), then \eqref{eq: main inequality} holds trivially. At this point, we emphasize that \(\tau_D = \infty\) always implies \(\tau_D = \tau_{\overline{D}}\).
    
    Finally, maximizing \eqref{eq: main inequality} over all \(h \in \mathbb{H}\) and \(m \in \m\), we conclude that \(\overline{V}^0_{t, \omega} \leq V^0_{t, \omega}\), which proves that \((t, \omega) \in \T\).
\end{discussion}

We conclude this section with a detailed example, illustrating a possible application of Theorem~\ref{theo: main result}.

\begin{example}\label{ex:mem}
    As an example of the type of setting we have in mind for Theorem \ref{theo: main result}, we consider a one-dimensional controlled SDE with a memory component in the drift, for which we can explicitly identify the small-noise limit as well as the associated optimal controls. We point out that the memory component makes the one-dimensional problem path-dependent.
    
    Let $d = r = 1$, and $D = (-1,1)$.
    For a given \([1, 2]\)-valued progressively measurable process \(\lambda = (\lambda_t)_{t \geq 0}\), we consider the controlled SDE
    \begin{equation}
            dX^\lambda_t = \Big(1+\lambda_t\int_0^t X_s\,ds \Big) \,dt + \sqrt{\varepsilon}\,dW_t,\quad
            X^\lambda_0 = 0,
    \end{equation}
    and reward function $g(t) = t \wedge 1$.
    Note that, in particular, the Lipschitz and growth assumptions on the drift are satisfied as, for any $T>0$, $t \in [0, T]$, $\lambda \in [1,2]$, and $\omega, \omega' \in  \Omega$, we have
    \[
        \bigg|1+\lambda\int_0^t \omega (s)\,ds - 1-\lambda\int_0^t\omega' (s)\,ds\bigg| \leq 2T \|\omega - \omega'\|_t, 
        \qquad 
        \bigg|1+\lambda\int_0^t \omega (s)\,ds\bigg| \leq 2T(1+ \|\omega\|_t).
    \]
    By virtue of Remark~\ref{rem: other control frameworks}, we aim to use Theorem \ref{theo: main result} to show that
    \begin{equation}
        V^\varepsilon := \varepsilon\log \sup_{\lambda \in \cA} E\left[\exp\left(\frac{g(\tau_D (X^\lambda)}{\varepsilon}\right)\right] \xrightarrow{\varepsilon \searrow 0}
        \sup_{\substack{h\in L^2(\bR_+)\\\lambda \in \mathbb{L} }}\left\{(\tau_D(X^{\lambda, h})\wedge 1)-\frac12\int_0^\infty |h_t|^2\, dt\right\} =: V^0,
    \end{equation}
    where $\cA$ denotes the set of progressively measurable controls with values in $[1,2]$, we have \(\mathbb{L} := \{ \lambda \colon \bR_+ \to [1, 2] \text{ measurable}\}\), and $X^{\lambda,h}$ is the solution of the deterministic two-dimensional control system
    \begin{align} \label{eq: linear system}
        \begin{cases} dX^{\lambda,h}_t = ( 1+ \lambda_t\mu^{\lambda,h}_t + h_t ) \, dt, & X^{\lambda,h}_0 = 0,\\
        d\mu^{\lambda,h}_t = X^{\lambda,h}_t \, dt, & \mu^{\lambda,h}_0 = 0.\end{cases} 
    \end{align}
    We now discuss the deterministic limiting control problem.
    Note that, due to the form of the payoff function $g$, all times $t>1$ do not influence the payoff.
    As the Cameron--Martin control in the limiting deterministic problem is costly, it is optimal to not exert any control after time $t=1$ and set $(h(s))_{s>1} \equiv 0$. Thus, we can equivalently optimize over $h \in L^2([0,1])$.

    Solving the system \eqref{eq: linear system} by variation of constants for a constant $\lambda \in [1,2]$, we find that, for any $h \in L^2([0,1])$,
    \begin{equation} \label{eq:example_explicit_X}
        X^{\lambda, h}_t = \lambda^{-1/2}\sinh(\sqrt{\lambda}t)+\int_0^t G^\lambda_t(s)h_s \,ds,
    \end{equation}
    where $G^\lambda_t(s) = \cosh(\sqrt{\lambda}(t-s))$ is the Green function associated with the controlled deterministic state $X^{\lambda, h}_t$.
    In particular, for the uncontrolled system, i.e., $h \equiv 0$, we have $X^{\lambda,0}_t = \lambda^{-1/2} \sinh(\sqrt{\lambda}t)$ and, as $\lambda \in [1,2]$,
    \[
    \tau^{\lambda, 0}_{D} : = \inf \{ s \geq 0 \colon X^{\lambda, 0}_s \not \in D \} = \lambda^{-1/2}\sinh^{-1}(\sqrt{\lambda}) \leq \tau^{1,0}_{D} = \sinh^{-1}(1) < 1,
    \]
    which is a lower bound for $V^0$.

    As \(L^2([0, 1])\) is a Hilbert space, for any $\lambda \in [1,2]$, $\theta \in (0,1]$, and $x \in \bR$, by projecting $0$ onto the affine hyperplane $\langle G^\lambda_\theta, h \rangle + \lambda^{-1/2} \sinh(\sqrt{\lambda}\theta) - x$, there exists a unique $h^{\lambda,\theta,x} \in L^2([0,1])$ minimizing $\tfrac12\|h\|^2_{L^2([0,1])}$ such that $X^{h^{\lambda,\theta,x}}_\theta = x$, which is explicitly given by
    \begin{equation}
        h^{\lambda, \theta,x}_s = 
        \begin{cases}
            \frac{x- \lambda^{-1/2}\sinh(\sqrt{\lambda}\theta)}{\| G^\lambda_\theta \|^2_{L^2([0,\theta])}}\cosh(\sqrt{\lambda}(\theta-s)), & 0<s<\theta,\\
            0,& \theta\leq s,
        \end{cases}
    \end{equation}
    where $\| G^\lambda_\theta \|^2_{L^2([0,\theta])}=\frac12\big(\theta+\frac{\sinh(2\sqrt{\lambda}\theta)}{2\sqrt{\lambda}}\big)$.
    Note that the total energy of the optimal control is given by
    \begin{equation}
        \tfrac12\|h^{\lambda,\theta,x}\|^2_{L^2([0,1])} = \frac12\int_0^1 |h^{\lambda,\theta,x}_t|^2\, dt = \frac{(x - \lambda^{-1/2}\sinh(\sqrt{\lambda}\theta))^2}{\theta+\tfrac12\lambda^{-1/2}\sinh(2\sqrt{\lambda}\theta)}.
    \end{equation}

    We first argue that it is sufficient to consider $\lambda \equiv 1$ and that the optimizing problem can be reduced to
    \begin{equation}\label{eq:example:intermediate_V0}
        V^0= \tau_D^{1,0} \vee \sup_{\substack{x \in \overline{D}\\\theta > 0}}\left\{(\tau_D(X^{h^{\theta, x}})\wedge 1)-\frac{(x - \sinh(\theta))^2}{\theta+\tfrac12\sinh(2\theta)}\right\},
    \end{equation}
    where $h^{\theta, x} \coloneqq h^{1, \theta, x}$.
    As $\lambda \equiv 1 \in \cA$, the RHS of equation \eqref{eq:example:intermediate_V0} is a lower bound. For the reverse inequality, consider that the total cost of the Cameron--Martin control for fixed $\theta \in (0,1]$ and $x \in \overline{D}$ is strictly increasing in $\lambda$. Approximating controls in $\cA$ by piecewise constant controls, we find that, for any fixed $\theta \in (0,1]$ and $x \in \overline{D}$, the control $\lambda \equiv 1$ generates an upper bound.
    Intuitively, the larger $\lambda$ is chosen, the more the Cameron--Martin control needs to work against the memory of the drift, which is costly. As choosing the control $\lambda$ does not incur any costs, it is optimal to choose it such that the accumulation of memory is minimal, i.e., one chooses $\lambda \equiv 1$.
    Consequently, we can reduce the optimizing problem to the one in equation \eqref{eq:example:intermediate_V0} and reduce the notation accordingly.

    We next work towards showing that the $x$-optimizer of the above equation is $+1$ and start by showing that any $x$-optimizer needs to lie on the boundary of $D$, i.e., $x^* \in \{-1,+1\}$.
    Indeed, assume that $h' \in L^2([0,1])$ is such that $\tau_D (X^{h'})>1$, which implies that $-1<X^{h'}_1 < 1$. As the norm of the optimal control $\tfrac12\|h^{1,x}\|^2_{L^2([0,1])}$ is decreasing in $x$ on $(-\infty, \sinh(1)]$ and the payoff for $X^{h'}$ is $1$, we find that
    \[
        (\tau_D(X^{h'})\wedge 1) - \tfrac12\|h'\|^2_{L^2([0,1])} < (\tau_D(X^{h^{1,1}})\wedge 1) - \tfrac12\|h^{1,1}\|^2_{L^2([0,1])}.
    \]
    Moreover, note that, for any $\theta \in (0,1]$, we have
    \[
    \tfrac12\|h^{\theta,1}\|^2_{L^2([0,1])} \leq \tfrac12\|h^{\theta,-1}\|^2_{L^2([0,1])},
    \]
    i.e., controlling into $+1$ is always cheaper than into $-1$, while the payoff is the same. 
    Lastly, it is easy to verify that using $h^{\theta,1}$ as a control, $X^{h^{\theta,1}}$ does not exhibit `overshooting behavior', i.e., for all $t \in [0,\theta)$, we have that $0\leq X^{h^{\theta,1}}_t <1$. Consequently, the $x$-optimizer in \eqref{eq:example:intermediate_V0} is given by $x^*=1$.
    Thus, we can reduce the optimization problem in equation \eqref{eq:example:intermediate_V0} further to a one-dimensional optimization problem and get
    \begin{equation}
        V^0 = \tau_D^{1,0} \vee\sup_{\theta>0} \bigg\{ (\theta\wedge1) - \frac{(1 - \sinh(\theta))^2}{\theta+\tfrac12\sinh(2\theta)}\bigg\}.
    \end{equation}
    
    Straightforward calculations now show that the supremum is attained at $\theta^*=1$ and greater than $\tau_D^{1,0}$, such that we find 
    \[
        V^0 = 1- \frac{(\sinh(1)-1)^2}{1 + \tfrac12 \sinh(2)}.
    \]

    To apply Theorem \ref{theo: main result} in this setting, it remains to verify that $(0,0) \in \T$. To this end, let $\delta> 0$ and consider the $\delta$-blow-up of $D$,
    \[
        D_\delta = [-(1+\delta), (1+\delta)].
    \]
    The aim now is to show that
    \begin{equation}
        V^{0,\delta} := \sup_{h\in L^2([0,1])}\left\{(\tau_{D_\delta}(X^h)\wedge 1)-\frac12\int_0^1 |h_t|^2\, dt\right\} \xrightarrow{\delta\searrow0} V^0.
    \end{equation}

    Repeating the arguments from above for $\delta< \sinh(1)-1$, we find that the optimal value is given by
    \[
        V^{0,\delta} = 1- \frac{(\sinh(1)-(1+\delta))^2}{1 + \tfrac12 \sinh(2)} \to V^0
    \]
    as $\delta\searrow0$. Consequently, \((0, 0) \in \T\).
\end{example}

The remainder of this paper is dedicated to the proof of our main result, Theorem~\ref{theo: main result}. Its proof is split into several steps that are detailed in the following sections. To give a short outline, in Theorem~\ref{theo: VF}, we establish a variational formula and, throughout Section~\ref{sec: convergence relaxed}, investigate convergence properties of associated relaxed control rules. The proof of Theorem~\ref{theo: main result} is then given in Section~\ref{sec: proof main}.

\section{A Variational Formula for Log-Transformed Control Problems} \label{sec: variational-formula}

In this section, we derive a variational formula for the entropic transformation \(\log \cE^\varepsilon (e^{\varphi})\) that is the key mathematical tool for our proof of Theorem~\ref{theo: main result}. 
Without loss of generality, we only consider the case \(\varepsilon = 1\), writing \(\cE \equiv \cE^1\) and
\[
\cE^*_t (\varphi) (\omega) := \log \cE_t( e^\varphi ) (\omega)
\]
for all \((t, \omega) \in \bR_+ \times \Omega\) and all upper semianalytic functions \(\varphi \colon \Omega \to \bR\).
The variational representation for \(\cE^*\) is based on the following optimal control problem:
Let \(\m^*\) be the set of all Radon measures on \(\bR_+ \times \Lambda \times \bR^r\), whose projections to \(\bR_+\) coincide with the Lebesgue measure. 
We endow \(\m^*\) with the vague topology, which turns it into a Polish space. 
The coordinate map on \(\m^*\) is denoted by \(M^*\).
Similar to the definition of the set \(\cK\), for \((t, \omega) \in \bR_+ \times \Omega\), let \(\cR (t, \omega)\) be the set of all relaxed control rules, with initial values \((t, \omega)\), 
associated to the operator
\begin{align*} 
 	\widetilde{L} (x, s, \omega, \lambda, z, \varphi) := \langle \mu (s, \omega, \lambda) + \sigma (s, \omega, \lambda) z, \nabla \varphi (x) \rangle + \tfrac{1}{2} \on{tr} \big[ \sigma (s, \omega, \lambda) \sigma^* (s, \omega, \lambda) \, \nabla^2 \varphi (x) \big].
 \end{align*}
 The elements of \(\cR (t, \omega)\) are probability measures on \((\Omega \times \m^*, \cF \otimes \M^*)\).
The control problem associated to these control rules is given by
\[
\cE'_t (\varphi) (\omega) := \sup_{P \in \cR (t, \omega)} E^P \Big[ \varphi (\X) - \frac{1}{2} \int_0^\infty \int_{\Lambda \times \bR^r} \| z \|^2 \, M^* (dt, d\lambda, dz) \Big],
\]
where \(\varphi \colon \Omega \to \bR\) is an upper semianalytic function. Let \(\USA_b (\Omega; \bR)\) denote the set of all {\em bounded} upper semianalytic functions.

\begin{theorem} \label{theo: VF} Let Assumptions~\ref{SA: control}~\ref{item:SA:continuity in control}--\ref{item:SA:growth} hold.
    For all \(\varphi \in \USA_b (\Omega; \bR)\) and \((t, \omega) \in \bR_+ \times \Omega\), we have 
\begin{align} \label{eq: main VF}
\cE^*_t (\varphi) (\omega) = \cE'_t (\varphi) (\omega).
\end{align} 
\end{theorem}

\begin{proof}
   The statement follows from Lemmas~\ref{lem: UC reduction VF} and \ref{lem: VF on lip} below. 
\end{proof}

Relating Theorem \ref{theo: VF} to existing literature, it extends \cite[Theorem~5.2]{CK_25} from a Markovian to a path-dependent setting. Furthermore, it can be viewed as an extension of the celebrated Bou\'e--Dupuis formula \cite{BD_98} to a controlled SDE setting with path-dependent coefficients.

\subsection{Reduction to \texorpdfstring{\(\Lip_b\)}{bounded Lipschitz} functions}

The first step in our proof for Theorem~\ref{theo: VF} is the reduction from general upper semianalytic to Lipschitz continuous input functions. 

\begin{lemma} \label{lem: UC reduction VF}
Let Assumptions~\ref{SA: control}~\ref{item:SA:continuity in control}--\ref{item:SA:growth} hold.
Then, the variational representation \eqref{eq: main VF} holds for all \(\varphi \in \USA_b (\Omega; \bR)\) if and only if it holds for all \(\varphi \in \Lip_b
(\Omega; \bR)\) satisfying \(\varphi (\omega) = \varphi (\omega (\, \cdot \wedge T))\) for some \(T = T_\varphi > 0\).
\end{lemma} 

\begin{proof}
It follows precisely as in the proofs of \cite[Theorems~2.21, 5.2]{CK_25} that \(\cE^*\) and \(\cE'\)
are convex expectations on the path space as defined in \cite[Definition~2.1]{CK_25}. Equipped with this observation, the claimed equivalence follows directly from the comparison \cite[Theorem~2.12]{CK_25}, taking also \cite[Remark~2.11]{CK_25} into consideration.
\end{proof}

\begin{remark}
We emphasize that Lemma~\ref{lem: UC reduction VF} reduces the identification of a variational formula for input functions depending on the {\em infinite} time horizon to the identification for functions with a {\em finite} time horizon. This reduction is made possible by the comparison result from \cite{CK_25} that traces the problem to finite-dimensional distributions, which are intrinsically defined on finite, but arbitrary, time intervals. 
\end{remark}

\subsection{Identification via PPDE techniques}
In view of Lemma~\ref{lem: UC reduction VF}, Theorem~\ref{theo: VF} follows once we establish the variational formula \eqref{eq: main VF} for all bounded Lipschitz continuous input functions \(\varphi\). 
In this section, we establish this result using recent advances in the theory of PPDEs. We first introduce a suitable solution concept. To reduce the technical flavor of this section, we delegated the definition of the set \(C^{1,2}_{pol}\) of test functions and the concepts of horizontal and vertical derivatives, gradients, and Hessians to Appendix~\ref{app: test}. Throughout, let \(T > 0\) be a finite time horizon and recall the definition of a stochastic interval \(\of 0, \tau \gs := \{ (t, \omega) \in \bR_+ \times \Omega : 0 \leq t \leq \tau (\omega) \}\) for a stopping time \(\tau \colon \Omega \to [0, \infty]\). Open or half-open stochastic intervals are defined analogously. 
For \((t, \omega,\phi) \in \of 0, T \gs \times C^{1,2}(\of 0, T \gs ; \bR ) \), we define
\begin{multline*}
	G(t, \omega, \nabla \phi(t,\omega), \nabla^2 \phi(t,\omega)) := \sup \Big\{ \langle \mu (t, \omega, \lambda), \nabla \phi (t, \omega) \rangle\\
	+ \tfrac{1}{2} \on{tr} \big[ \sigma (t, \omega, \lambda) \sigma^* (t, \omega, \lambda) \, \nabla^2\phi (t, \omega) \big]\colon \lambda \in \Lambda \Big\},
\end{multline*}
and 
\begin{multline*}
	\widetilde{G}(t, \omega, \nabla \phi(t,\omega), \nabla^2 \phi(t,\omega)) := \sup \Big\{ \langle \mu (t, \omega, \lambda) + \sigma (t, \omega, \lambda) h, \nabla \phi (t, \omega) \rangle
	\\+ \tfrac{1}{2} \on{tr} \big[ \sigma (t, \omega, \lambda) \sigma^* (t, \omega, \lambda) \, \nabla^2\phi (t, \omega) \big] - \tfrac{1}{2} \| h \|^2 \colon \lambda \in \Lambda, h \in \mathbb{R}^r \Big\}.
\end{multline*}
For either \(H = G\) or \(H = \widetilde{G}\), we consider the following backward PPDE, which we call \(H\)-backward equation with terminal value \(\psi\), 
\begin{equation} \label{eq: PIDE}
	\begin{cases}   
		\p f (t, \omega) + H (t, \omega, \nabla f(t,\omega), \nabla^2 f(t,\omega)) = 0, & \text{for } (t, \omega) \in[0, T) \times \Omega , \\
		f (T, \omega) = \psi (\omega), & \text{for } \omega \in \Omega,
	\end{cases}
\end{equation}
where \(\psi \in C_b (\Omega; \bR)\) satisfies \(\psi (\omega) = \psi (\omega (\cdot \wedge T))\).

As one cannot expect to solve the equations above in the classical sense, we need to define an appropriate weaker solution concept. A prominent concept considered in the literature, see, e.g., \cite{cosso_russo_22,Z_23}, are Crandall--Lions-type viscosity solutions, the definition of which is the following.

\begin{definition}\label{definition:vis_sol}
    Consider the backward PPDE in equation \eqref{eq: PIDE}.
    \begin{enumerate}[(a)]
        \item We say that a continuous function \(u \colon \of 0, T\gs \to \urange\) is a {\em viscosity subsolution} to \eqref{eq: PIDE} if \(u(T, \cdot) \leq \psi\) and, for any \((t,\omega) \in \of 0, T \of \) and \( \phi \in C^{1,2}_{pol}(\of t, T \gs ; \urange )\) satisfying
    	\begin{equation}
    	    0 = (u-\phi)(t,\omega) = \sup \{ (u-\phi)(s,\omega') \colon (s,\omega') \in \of t, T \gs  \},
    	\end{equation}
    	we have
    	\(
    		\p \phi (t, \omega) + H (t, \omega, \nabla \phi(t,\omega), \nabla^2 \phi (t,\omega)) \geq 0.
    	\)

        \item We say that a continuous function \(u \colon \of 0, T\gs \to \urange\) is a \emph{viscosity supersolution} to \eqref{eq: PIDE} if \(u(T, \cdot) \geq \psi\) and, for any \((t,\omega) \in \of 0, T \of \) and \( \phi \in C^{1,2}_{pol}(\of t, T \gs ; \urange )\) satisfying
    	\begin{equation}
    	    0 = (u-\phi)(t,\omega) = \inf \{ (u-\phi)(s,\omega') \colon (s,\omega') \in \of t, T \gs  \},
    	\end{equation}
    	we have
    	\(
    		\p \phi (t, \omega) + H (t, \omega, \nabla \phi(t,\omega), \nabla^2 \phi (t,\omega)) \leq 0.
    	\)

        \item We say that a continuous function \(u \colon \of 0, T\gs \to \urange\) is a \emph{viscosity solution} to \eqref{eq: PIDE} if it is both a viscosity sub- and supersolution.
    \end{enumerate}
\end{definition}

For \((t, \omega), (t', \omega') \in \of 0, T\gs\), we define 
\begin{align} \label{eq: def d metric}
\mathsf{d} ( (t, \omega); (t', \omega') ) := \big( 1 + \| \omega \|_t + \| \omega'\|_{t'} \big) |t - t'|^{\frac{1}{2}} + \| \omega (\, \cdot  \wedge t) - \omega' (\, \cdot \wedge t') \|_T. 
\end{align} 
We say that a function \(v \colon \of 0, T \gs \to \bR\) is \(\mathsf{d}\)-Lipschitz continuous if there exists a constant \(L > 0\) such that 
\[
| v (t, \omega) - v (t', \omega') | \leq L \mathsf{d} ((t, \omega); (t', \omega') )
\]
for all \((t, \omega), (t', \omega') \in \of 0, T \gs\).

\begin{remark}\label{remark:zhou_comparison}
    Note that given Assumption \ref{SA: control} and, in particular, the Lipschitz and growth assumptions in Parts \ref{item:SA:Lipschitz} and \ref{item:SA:growth}, \cite[Theorem 6.1]{Z_23} provides a comparison principle for viscosity solutions to the \(G\)-backward equation with terminal value \(\psi\) in equation \eqref{eq: PIDE}. This in turn implies uniqueness of viscosity solutions in the class of \(\mathsf{d}\)-Lipschitz continuous functions with linear growth, cf.\ \cite[Theorem 6.2]{Z_23}.
\end{remark}

\begin{lemma} \label{lem: value solution} 
Let Assumptions~\ref{SA: control}~\ref{item:SA:continuity in control}--\ref{item:SA:growth} hold and let \(\psi \in \Lip_b (\Omega; \bR)\). 
\begin{enumerate}[(a)]
\item\label{item:value_solution:lhs} The function \(u := \cE_t (\psi) (\omega)\), \((t, \omega) \in \of 0, T \gs\), is the unique bounded \(\mathsf{d}\)-Lipschitz continuous viscosity solution to the \(G\)-backward equation with terminal value \(\psi\).

\item\label{item:value_solution:rhs} The function 
\[
\tilde{v} (t, \omega) := \sup_{P \in \cR (t, \omega)} E^P \Big[ \psi (\X) - \frac{1}{2} \int_0^\infty \int_{\Lambda \times \bR^r} \| z \|^2 \, M^* (dt, d\lambda, dz) \Big], \quad (t, \omega) \in \of 0, T \gs, 
\]
is a bounded \(\mathsf{d}\)-Lipschitz continuous viscosity solution to the \(\widetilde{G}\)-backward equation with terminal value \(\psi\). 
\end{enumerate} 
\end{lemma} 

\begin{remark}
    In case the domain of uncertainty \(\{ (\mu (t, \omega, \lambda), \sigma \sigma^* (t, \omega, \lambda)) : \lambda \in \Lambda \}\) is convex, part~\ref{item:value_solution:lhs} follows directly from \cite[Theorem~5.1]{CN_25_JEEQ} and \cite[Theorem~4.13]{CN_23_EJP}.
\end{remark}

\begin{proof}[Proof of Lemma \ref{lem: value solution}]
For part \ref{item:value_solution:lhs}, it follows almost verbatim as in the proof of \cite[Theorem~4.13]{CN_23_EJP} that \(u\) is the unique bounded \(\mathsf{d}\)-Lipschitz continuous viscosity solution to the \(G\)-backward equation with terminal value~\(\psi\).  

\smallskip

We show \ref{item:value_solution:rhs}, i.e., that \(\tilde{v}\) is a bounded \(\mathsf{d}\)-Lipschitz continuous viscosity solution to the \(\widetilde{G}\)-backward equation with terminal value \(\psi\). We will adapt arguments from the proof of \cite[Theorem~4.12]{CN_23_EJP} to our situation with unbounded action space. The main technical difficulty is related to the fact that our concept of viscosity solution uses test functions of arbitrary polynomial growth, while the controlled processes under consideration may only have a second moment. We use the stopping argument in Lemma \ref{lem: stopped moment bound} below to solve this problem. 
To shorten our discussion, we only detail the subsolution property, leaving the supersolution case and the straightforward verification of the \(\mathsf{d}\)-Lipschitz continuity to the reader. 
By the definition of the relaxed control rules $\cR (t, \omega)$, we have \(\tilde{v} (T, \omega) = \psi (\omega)\). To show the subsolution property, take \((t,\omega) \in \of 0,T \of \) and \( \phi \in C^{1,2}_{pol}(\of t,T \gs ; \bR )\) satisfying
\[
	0 = (\tilde{v} - \phi)(t,\omega) = \sup \{ (\tilde{v} - \phi)(s,\omega') \colon (s,\omega') \in \of t,T \gs \}.
\]
For a fixed \(R > 0\), set 
\[
	\rho := \inf \Big\{ s \geq 0 \colon \int_t^{s + t} \int_{\Lambda \times \bR^r} \| z \|^{3/2} \, M^* (dr, d \lambda, dz) \geq R \Big\}. 
\] 
The proof of the following technical stopping lemma is postponed after the proof of Lemma~\ref{lem: value solution}.
\begin{lemma} \label{lem: stopped moment bound}
    Let Assumptions~\ref{SA: control}~\ref{item:SA:continuity in control}--\ref{item:SA:growth} hold and take \(p \in \mathbb{N},\, T > 0,\, t \in \mathbb{R}_+\), and \(\omega \in \Omega\).
	Then, there exists a constant \(C > 0\), only depending on $t$, $\omega$, $R$, $p$, $T$, and the linear growth constants from Assumption~\ref{SA: control}~\ref{item:SA:growth}, such that 
        \[
        \sup_{P \in \cR (t, \omega)} E^P \Big[ \| \X \|_{T \wedge \rho\, + \, t}^p \Big] \leq C.
        \]
\end{lemma}
	
Now, fix \(0 < u < T - t\). By dynamic programming, see \cite[Theorem~3.4]{ElKa15}, we have
	\begin{align} 
		0 &= \sup_{P \in \cR (t, \omega)} E^P \Big[ \tilde{v} (u  \wedge \rho + t, X)  -  \tilde{v} (t, \omega) - \frac{1}{2} \int_t^{u \wedge \rho + t} \int_{\Lambda \times \bR^r} \|z\|^2 \, M^* (ds, d\lambda, dz) \Big] 
        \\&= \sup_{P \in \cR_M (t, \omega)} E^P \Big[ \tilde{v} (u  \wedge \rho + t, X)  -  \tilde{v} (t, \omega) - \frac{1}{2} \int_t^{u \wedge \rho + t} \int_{\Lambda \times \bR^r} \|z\|^2 \, M^* (ds, d\lambda, dz) \Big] 
		\\&\leq \sup_{P \in \cR_M(t, \omega)} E^P \Big[ \phi (u \wedge \rho + t,X) - \phi (t, \omega) - \frac{1}{2} \int_t^{u \wedge \rho + t} \int_{\Lambda \times \bR^r} \|z\|^2 \, M^* (ds, d\lambda, dz) \Big],
        \label{eq: visco sub conseq of DPP}
	\end{align}
    where 
    \[
    \cR_M (t, \omega) := \Big\{ P \in \cR (t, \omega) : E^P \Big[\int_0^\infty \int_{\Lambda \times \bR^r} \| z \|^2 \, M^* (ds, d\lambda, dz) \Big] \leq M \Big\}
    \]
    with \(M = 8 \| \psi \|_\infty\) and, for the second equality, we used the fact that the optimization ignores measures from \(\cR (t, \omega) \setminus \cR_M (t, \omega)\).
	Now fix \(P \in \cR_M (t, \omega)\). Using the functional  It\^o formula \cite[Proposition~7]{CF_10}, we obtain that 
	\begin{equation*} \begin{split}
			\phi (u + t&, X) - \phi (t, \omega) 
			\\&= \int_t^{u + t} \p \phi (s, X)  \, d s  + \frac{1}{2} \int_t^{u + t} \int_{\Lambda \times \bR^r} \on{tr} \big[ \sigma \sigma^* (s,X, \lambda)  \nabla^2 \phi (s,X) \big] \, M^* (ds, d \lambda, d z)
			\\&\hspace{0.5cm} + \int_t^{u + t} \langle \nabla \phi (s,X) , d X_{s} \rangle.
		\end{split}
	\end{equation*}
	By virtue of the linear growth condition, the polynomial growth of \(\nabla \phi\) and Lemma~\ref{lem: stopped moment bound}, when stopped at \(\rho\), the local martingale part of the stochastic integral above is a true martingale and we find
	\begin{align*}
		E^P \Big[ &\int_t^{u \wedge \rho + t} \langle \nabla \phi (s, X), d X_s\rangle \Big] 
        = E^P \Big[ \int_t^{u \wedge \rho + t} \langle \nabla \phi (s, X), \mu (s, X, \lambda) + \sigma (s, X, \lambda) z \rangle \, M^* (ds, d\lambda, dz) \Big],
	\end{align*}
	which implies that
	\begin{equation} \label{eq: after martingale part vanish}
		\begin{split}
			E^P \Big[ \phi (u \wedge \rho &+ t,  X) - \phi (t, \omega) - \frac{1}{2} \int_t^{u \wedge \rho + t} \int_{\Lambda \times \bR^r} \|z\|^2 \, M^* (ds, d\lambda, dz)\Big] 
			\\&= E^P \Big[ \int_t^{u \wedge \rho + t} \p \phi (s, X ) \, d s + \int_t^{u \wedge \rho + t} \widetilde{L} (s, \lambda, z) \, M^* (ds, d \lambda, dz)  \Big], 
		\end{split}
	\end{equation}
	where
	\[
	\widetilde{L} (s, \lambda, z) := \langle \nabla \phi (s, X), \mu (s, X, \lambda) + \sigma (s, X, \lambda) z \rangle + \frac{1}{2} \on{tr} \big[ \sigma \sigma^* (s,X, \lambda)  \nabla^2 \phi (s,X) \big] - \frac{1}{2} \|z\|^2.
	\]
	Set 
	\[
	\mathfrak{G} (u, X) := \p \phi (u + t, X) + \widetilde{G}(u + t, X, \nabla \phi(u + t,X), \nabla^2 \phi(u + t,X)).
	\] 
	From equations \eqref{eq: visco sub conseq of DPP} and \eqref{eq: after martingale part vanish}, we get that 
	\begin{align}\label{eq: subsolution main ineq}
		0 \leq \frac{1}{u} \int_0^u \sup_{P \in \cR_M(t, \omega)} E^P \big[ \mathfrak{G} (s, X) \1_{\{s \leq \rho \}} \big] \, ds, \quad u > 0.
	\end{align}
	We now show that 
	\begin{align} \label{eq: to show continuous} 
	s \mapsto \sup_{P \in \cR_M(t, \omega)} E^P \big[ \mathfrak{G} (s, X) \1_{\{s \leq \rho \}} \big]
	\end{align} 
	is continuous in zero. As a consequence, Lebesgue's differentiation theorem yields that 
	\[
	0 \leq \lim_{u \searrow 0} \frac{1}{u} \int_0^u \sup_{P \in \cR_M(t, \omega)} E^P \big[ \mathfrak{G} (s, X) \1_{\{s \leq \rho \}} \big] \, ds = \sup_{P \in \cR_M(t, \omega)} E^P \big[ \mathfrak{G} (0, X) \big] = \mathfrak{G} (0, \omega), 
	\] 
	which proves that \(\tilde{v}\) has the subsolution property. It remains to show that the map \eqref{eq: to show continuous} is continuous in zero. 
    We define 
    \[
    \m' := \Big\{ m \in \mathbb{M}^* : \int_0^T \int_{\Lambda \times \bR^r} \|z\|^{3/2} \, m (ds, d\lambda, dz) < \infty \, \text{ for all } T > 0 \Big\}
    \]
    and endow this space with the local \(\frac{3}{2}\)-Wasserstein topology, i.e., the weakest topology such that, for every \(T > 0\), the map \(M \mapsto \1_{[0, T]} (s) \, M(ds, d\lambda, dz)\) is continuous in the \(\frac{3}{2}\)-Wasserstein topology. Note that the choice of the power \(3/2\) is quite arbitrary. Any power from \((1, 2)\) may serve our purpose. 
    Thus, for every \(s \in \bR_+\), the map
    \[
    M \mapsto \int_t^{s + t} \int_{\Lambda \times \bR^r} \|z\|^{3/2} \, M (dr, d\lambda, dz)
    \]
    is continuous on \(\m'\).
    Hence, the set 
	\[
	\{ \rho \leq s \} =  \Big\{ \int_t^{s + t} \int_{\Lambda \times \bR^r} \|z\|^{3/2} \, M^* (dr, d \lambda, dz) \geq R \Big\}
	\] 
	is closed in \(\m'\) and consequently, \(\rho\) is lower semicontinuous on \(\m'\). Let \((s_n, m_n)_{n = 1}^\infty \subseteq (0, \infty) \times \m'\) be a sequence such that \((s_n, m_n) \to (0, m)\). Recall that the set \(A := \{ m_n \colon n \in \mathbb{N}\} \cup \{m \}\) is compact in~\(\m'\). Then, as lower semicontinuous functions attain their minimal values on compact sets, there exists an \(m^* \in A\) such that 
    \[
    \inf_{m \in A} \varrho (m) = \varrho (m^*) > 0, 
    \]
    where we use that \(\varrho\) is a strictly positive function. 
	Using this fact, it follows that \[\lim_{n \to \infty} \1_{\{s_n \, \leq \, \rho (m_n)\}} = 1.\]
    As a consequence, using Assumption \ref{SA: control}, we obtain that
	\[
\lim_{n \to \infty} \mathfrak{G} (s_n, \alpha_n) \1_{\{s_n \leq \rho (m_n)\}} = \mathfrak{G} (0, \alpha)
	\] 
	for every sequence \((\alpha_n)_{n = 1}^\infty \subseteq \Omega\) with \(\alpha_n \to \alpha\). 
    In the following, we consider the set \(\cR_M (t, \omega)\) as a subset of probability measures on \(\Omega \times \m'\).
    Using that \(\mathfrak{G}\) grows at most polynomially and Lemma~\ref{lem: stopped moment bound}, it follows that the map 
	\[
	(s, P) \mapsto E^P \big[ \mathfrak{G} (s, X) \1_{\{s \leq \rho \}} \big]
	\] 
	is continuous at any point of \(\{0\} \times \cR_M (t, \omega)\). Thus, by the compactness of \(\cR_M (t, \omega)\), which follows similar to the proof of Proposition~\ref{prop: 2 step to Hausdorff}~\ref{item:prop: 2 step to Hausdorff:a} below, enhancing the set \eqref{eq: sets in M compact} through \cite[Corollary~5.6]{CD_18_I} to account for the local Wasserstein topology on \(\m'\), and
    Berge's maximal theorem, see, e.g., \cite[Theorem~1.3.4]{hu_papa}, we conclude that the map \eqref{eq: to show continuous} is continuous in zero. This completes the proof of the subsolution property. As mentioned above, we skip the proof of the supersolution property and the \(\mathsf{d}\)-Lipschitz continuity, noting that it proceeds along similar lines as the proofs of \cite[Theorems~4.12 and 4.13]{CN_23_EJP}.
\end{proof} 

\begin{proof}[Proof of Lemma~\ref{lem: stopped moment bound}]
    The claim follows from a standard Gronwall argument. Without loss of generality, we assume that \(p \geq 3\). Take \(P \in \cR (t, \omega)\) and set 
    \[
        \xi := \rho \wedge \inf \{ s \geq 0 \colon \| X \|_{s \wedge \rho + t} \geq \ell \}, \quad \ell > 0.
    \]
    By H\"older's inequality, the Burkholder--Davis--Gundy inequality, and the linear growth conditions from Assumption~\ref{SA: control}~\ref{item:SA:growth}, we obtain that, for all \(\beta \in [0, T]\), 
    \begin{align*}
        E^P \Big[ \| \X \|_{\beta \wedge \xi\, + \, t}^p \Big] & \leq 3^{p - 1} \Big(\| \omega \|_{t}^p + E^P \Big[ \Big( \int_{t}^{\beta \wedge \xi + t} \int_{\Lambda \times \bR^r}\| \mu (s, X, \lambda) + \sigma (s, X, \lambda) z \| \, M^* (ds, d\lambda, dz) \Big)^p \\&\hspace{4.5cm}+ \Big( \int_t^{\beta \wedge \xi + t} \int_{\Lambda \times \bR^r} \| \sigma (s, X, \lambda) \|^2 \, M^* (ds, d\lambda, dz) \Big)^{p / 2} \Big] \Big)
        \\&\leq C \Big( 1 +  T \, E^P \Big[ \int_t^{\beta \wedge \xi + t}\int_{\Lambda \times \bR^r} \| \mu (s, X, \lambda) \|^p \, M^* (ds, d \lambda, dz) \Big] 
        \\&\hspace{0.5cm}+ E^P \Big[ \Big( \int_t^{\beta \wedge \xi + t}\int_{\Lambda \times \bR^r} \| \sigma (s, X, \lambda) \|^3 \, M^* (ds, d\lambda, dz) \Big)^{p/3} \\&\hspace{4cm} \cdot \Big( \int_t^{\beta \wedge \xi + t} \int_{\Lambda \times \bR^r}\| z \|^{3/2} \, M^* (ds, d \lambda, dz) \Big)^{2p / 3}\Big] 
        \\&\hspace{0.5cm}+ \Big( \int_t^{\beta \wedge \xi + t} \int_{\Lambda \times \bR^r}\| \sigma (s, X, \lambda) \|^2 \, M^* (ds, d\lambda, dz) \Big)^{p / 2} \Big] \Big) 
        \\&\leq C \Big( 1 + \int_0^{\beta} E^P \big[ \| X \|_{s \wedge \xi + t}^p \big] \, ds \Big),
    \end{align*}
    where the constant \(C > 0\) only depends on \(t, \omega, R, p, T\) and the linear growth constants from Assumption~\ref{SA: control}~\ref{item:SA:growth} and changed from the second to third inequality.
    From this point, applying Gronwall's lemma and taking \(\ell \to \infty\) implies the claim. 
\end{proof}

\begin{lemma} \label{lem: VF on lip}
Let Assumptions~\ref{SA: control}~\ref{item:SA:continuity in control}--\ref{item:SA:growth} hold.
For all \(\psi \in \Lip_b (\Omega; \bR)\) and~\((t, \omega) \in \of 0, T \gs\), 
\[
\tilde{v} (t, \omega) = \log \cE_t (e^{\psi}) (\omega)
\]
with \(\tilde{v}\) as in Lemma~\ref{lem: value solution}~\ref{item:value_solution:rhs}.
\end{lemma} 
\begin{proof}
    Thanks to Lemma~\ref{lem: value solution}~\ref{item:value_solution:rhs}, 
    the function \(\tilde{v}\) is a bounded \(\mathsf{d}\)-Lipschitz continuous viscosity solution to the \(\widetilde{G}\)-backward equation with terminal value \(\psi\). In the following, we show that \(v := e^{\tilde{v}}\) is a bounded \(\mathsf{d}\)-Lipschitz continuous viscosity solution to the \(G\)-backward equation with terminal value \(e^\psi\). As \(u (t, \omega) = \cE_t (e^\psi) (\omega)\) is the unique such viscosity solution by Lemma~\ref{lem: value solution}~\ref{item:value_solution:lhs}, we may conclude that \(u = v\), which is the claim of Lemma~\ref{lem: VF on lip}. 
    In the remainder of the proof we show that \(v\) solves the \(G\)-backward equation. 

    {\em Step 1:} We start with the boundedness and the regularity. First, as \(\tilde{v}\) is bounded, \(v\) is trivially bounded as well. Moreover, using again that \(\tilde{v}\) is bounded, together with the local Lipschitz continuity of the exponential function, the \(\mathsf{d}\)-Lipschitz continuity transfers also from \(\tilde{v}\) to \(v\).

    {\em Step 2:} We now verify the subsolution property.
    Let $\phi \in C^{1,2}_{pol}(\of t, T \gs ; \bR)$ be such that
    \begin{equation}
        0 = (v - \phi)(t,\omega) = \sup \{ (v - \phi)(s,\omega') \colon (s,\omega') \in \of t, T \gs  \}.
    \end{equation} 
Notice that \(v \leq \phi\), which implies that \(\phi\) is bounded away from zero. Hence, \(\tilde{\phi} := \log (\phi)\) is well-defined and of polynomial growth. In particular, \(\tilde{\phi} \in C^{1,2}_{pol}(\of t, T \gs ; \bR)\) with
 \begin{align}
        \p \log (\phi) (t, \omega) &= \lim_{h \searrow 0} \frac{ \log (\phi(t + h, \omega(\cdot \wedge t))) - \log (\phi(t, \omega(\cdot \wedge t)))}{\phi(t + h, \omega(\cdot \wedge t)) - \phi(t, \omega(\cdot \wedge t))} \cdot \frac{\phi(t + h, \omega(\cdot \wedge t)) - \phi(t, \omega(\cdot \wedge t))}{h}\\
        &= \frac{\p \phi(t, \omega)}{\phi(t, \omega)},
\\
        \partial_i \log (\phi) (t, \omega) &= \lim_{h \to 0} \frac{\log (\phi (t, \omega + h e_i\1_{[t,T]})) - \log (\phi (t, \omega))}{\phi (t, \omega + h e_i\1_{[t,T]}) - \phi(t,\omega)} \cdot
        \frac{\phi (t, \omega + h e_i\1_{[t,T]}) - \phi(t,\omega)}{h}\\
        &= \frac{ \partial_i \phi (t,\omega)}{\phi(t, \omega)},
    \end{align}
    and
    \begin{equation}
        \partial_j \partial_i \log (\phi) (t, \omega)  = -\frac{ \partial _j \phi (t, \omega) \partial _i \phi (t, \omega)}{\phi(t, \omega)^2} + \frac{\partial_j\partial_i \phi (t,\omega)}{\phi(t, \omega)} .
    \end{equation}
Moreover, we have 
\[
(\tilde{v} - \tilde{\phi}) (t, \omega) = 0, \quad \tilde{v} \leq \tilde{\phi}.
\]
Thus, \(\tilde{\phi}\) is a test function that satisfies the viscosity tangency condition for the function \(\tilde{v}\) at the point \((t, \omega)\). As \(\tilde{v}\) solves the \(\widetilde{G}\)-backward equation, this entails that 
\begin{align} \label{eq: tilde G subsolution}
    \partial_t \tilde{\phi} (t, \omega) + \widetilde{G} (t, \omega, \nabla \tilde{\phi} (t, \omega), \nabla^2 \tilde{\phi} (t, \omega)) \geq 0.
\end{align}
We now relate this statement to the \(G\)-equation. 
Namely, we obtain that 
\begin{align*}
    \widetilde{G} (t&, \omega, \nabla \tilde{\phi} (t, \omega), \nabla^2 \tilde{\phi} (t, \omega)) 
    \\&= \sup \Big\{ \langle \mu (t, \omega, \lambda) + \sigma (t, \omega, \lambda) h, \nabla \tilde{\phi} (t, \omega) \rangle
	\\&\hspace{3cm}+ \tfrac{1}{2} \on{tr} \big[ \sigma (t, \omega, \lambda) \sigma^* (t, \omega, \lambda) \, \nabla^2\tilde{\phi} (t, \omega) \big] - \tfrac{1}{2} \| h \|^2 \colon \lambda \in \Lambda, h \in \mathbb{R}^r \Big\}
    \\&= \sup \Big\{ \langle \mu (t, \omega, \lambda), \nabla \tilde{\phi} (t, \omega) \rangle + \tfrac{1}{2} \| \sigma^* (t, \omega, \lambda) \nabla \tilde{\phi} (t, \omega) \|^2 
	\\&\hspace{3cm}+ \tfrac{1}{2} \on{tr} \big[ \sigma (t, \omega, \lambda) \sigma^* (t, \omega, \lambda) \, \nabla^2\tilde{\phi} (t, \omega) \big] \colon \lambda \in \Lambda \Big\}
    \\&= \sup \Big\{ \frac{1}{\phi (t, \omega)} \langle \mu (t, \omega, \lambda), \nabla \phi (t, \omega) \rangle + \frac{1}{2 \phi (t, \omega)^2} \| \sigma^* (t, \omega, \lambda) \nabla \phi (t, \omega) \|^2 
	\\&\hspace{3cm}+ \frac{1}{2 \phi (t, \omega)} \on{tr} \big[ \sigma (t, \omega, \lambda) \sigma^* (t, \omega, \lambda) \, \nabla^2\phi (t, \omega) \big]
    \\&\hspace{3cm}- \frac{1}{2 \phi (t, \omega)^2} \| \sigma^* (t, \omega, \lambda) \nabla \phi (t, \omega) \|^2 \colon \lambda \in \Lambda \Big\}
    \\&= \frac{1}{\phi (t, \omega)} G (t, \omega, \nabla \phi (t, \omega), \nabla^2 \phi (t, \omega)).
\end{align*}
Consequently, as \(\phi (t, \omega) > 0\), \eqref{eq: tilde G subsolution} entails that 
\[
\partial_t \phi (t, \omega) + G (t, \omega, \nabla \phi (t, \omega), \nabla^2 \phi (t, \omega)) \geq 0, 
\]
showing that \(v\) is a viscosity subsolution to the backward equation \(G\). 

\smallskip
{\em Step 3:} We now shortly discuss the supersolution part. In principle, it works similarly to the subsolution part with one important exception: As in the supersolution case test functions lie below the candidate supersolution, we cannot immediately conclude that admissible test functions are bounded away from zero. This problem can be overcome by a straightforward cut-off argument of the test function with a smooth increasing cut-off function $\eta_\ell \colon \bR \to (0,\infty)$ such that $\eta_\ell(r) =r$ if $r\geq \ell $ and $\eta_\ell(r)= \tfrac{\ell}{2}$ if $r \leq \tfrac{\ell}{2}$ with $\ell \in (0,\inf v/2)$.
We omit the details here for brevity.  
\end{proof}

\section{Convergence of Relaxed Control Rules} \label{sec: convergence relaxed}
Adapting the terminology from the previous sections, let \(\cR^\varepsilon (t, \omega)\) be the set of all relaxed control rules associated with the operator 
\begin{align*} 
	\widetilde{L}^\varepsilon (x, s, \omega, \lambda, h, \varphi) := \langle \mu (s, \omega, \lambda) + \sigma (s, \omega, \lambda) h, \nabla \varphi (x) \rangle + \tfrac{\varepsilon}{2} \on{tr} \big[ \sigma (s, \omega, \lambda) \sigma^* (s, \omega, \lambda) \, \nabla^2 \varphi (x) \big],
\end{align*} 
and let 
\(\cR^0 (t, \omega)\) be the set of all relaxed control rules associated to 
\begin{align*} 
	\widetilde{L}^0 (x, s, \omega, \lambda, h, \varphi) := \langle \mu (s, \omega, \lambda) + \sigma (s, \omega, \lambda) h, \nabla \varphi (x) \rangle.
\end{align*} 
For \(M > 0\), we also set 
\begin{align*}
\cR^\varepsilon_M (t, \omega) &:= \Big\{ P \in \cR^\varepsilon (t, \omega) \colon E^P \Big[ \int_0^\infty \int_{\Lambda \times \bR^r} \| z \|^2 \, M^* (dt, d \lambda, dz) \Big] \leq M \Big\}, 
\\ 
\cR^0_M (t, \omega) &:= \Big\{ P \in \cR^0 (t, \omega) \colon E^P \Big[ \int_0^\infty \int_{\Lambda \times \bR^r} \| z \|^2 \, M^* (dt, d \lambda, dz) \Big] \leq M \Big\}.
\end{align*}

For two subsets \(A, B \subseteq E\) of a metric space \((E, d_E)\), the Hausdorff distance \(d_H (A, B)\) between \(A\) and \(B\) is defined as
\[
d_H (A, B) := \max \Big\{ \sup_{a \in A} d_E(a, B), \ \sup_{b \in B} d_E (b, A) \Big\}, 
\]
where \(d_E (a, B) := \inf_{b \in B} d_E (a, b)\). As usual, we say that \((A_n)_{n = 1}^\infty\) converges to \(A\) in the Hausdorff metric topology if \(d_H (A_n, A) \to 0\). For more information on this topology, we refer to \cite[Section~1.1.1]{hu_papa}. We report a simple fact, which appears to be well-known, although we did not find a reference.

\begin{lemma} \label{lem: HD convergence}
Let \((E, d_E)\) be a metric space, \((A_n)_{n = 1}^\infty, A \subseteq E\) and assume the following:
\begin{enumerate}[(a)]
\item\label{item:lem: HD convergence:rel_comp} Every sequence \((a_n)_{n = 1}^\infty\) with \(a_n \in A_n\) is relatively compact and all of its accumulation points are in \(A\).
\item\label{item:lem: HD convergence:subsub} The set \(A\) is compact and for every \(a \in A\) and every subsequence \((n_m)_{m = 1}^\infty\) there exists a further subsequence \((N_m)_{m = 1}^\infty \subseteq (n_m)_{m = 1}^\infty\) and a sequence \((a_{N_m})_{m = 1}^\infty\) with \(a_{N_m} \in A_{N_m}\) such that \(a_{N_m} \to a\).
\end{enumerate} 
Then, \(A_n \to A\) in the Hausdorff metric topology. 
\end{lemma} 
\begin{proof}
     We first show that \(\sup_{a \in A_n} d_E (a, A) \to 0\). For every \(n \in \mathbb{N}\), there exists an \(a_n \in A_n\) such that 
    \[
    \sup_{a \in A_n} d_E (a, A) \leq d_E (a_n, A) + \frac{1}{n}. 
    \]
    By Assumption \ref{item:lem: HD convergence:rel_comp}, the sequence \((a_n)_{n = 1}^\infty\) has an accumulation point \(\lim_{m \to \infty} a_{n_m} = a_0 \in A\). Thus, 
    \[
    \limsup_{n \to \infty} \sup_{a \in A_n} d_E (a, A)  \leq \limsup_{m \to \infty} d_E (a_{n_m},A) = d (a_0, A) = 0,
    \]
    which proves \(\sup_{a \in A_n} d_E (a, A) \to 0\). 

    We turn to the proof of \(\sup_{a \in A} d_E (a, A_n) \to 0\).
    There exists a sequence \((a_n)_{n = 1}^\infty \subseteq A\) with 
    \[
    \sup_{a \in A} d_E (a, A_n) \leq d_E (a_n, A_n) + \frac{1}{n}.
    \]
    Moreover, passing to a subsequence \((n_m)_{m = 1}^\infty\), we have 
    \[
    \limsup_{n \to \infty} \sup_{a \in A} d_E (a, A_n) \leq \lim_{m \to \infty} d_E (a_{n_m}, A_{n_m}).
    \] 
    As \(A\) is compact, there exists an \(a_0 \in A\) and a further subsequence \((n'_m)_{m = 1}^\infty \subseteq (n_m)_{m = 1}^\infty\) such that \(\lim_{m \to \infty} a_{n'_m} = a_0\). 
    By Assumption \ref{item:lem: HD convergence:subsub}, there exists another subsequence \((N_m)_{m = 1}^\infty \subseteq (n'_m)_{m = 1}^\infty\) and a sequence \((\bar{a}_{N_m})_{m = 1}^\infty\) with \(\bar{a}_{N_m} \in A_{N_m}\) such that \(\bar{a}_{N_m} \to a_0\). Finally, we have 
    \begin{align*}
    d_E (a_{N_n}, A_{N_n}) 
    &\leq d_E (a_{N_n}, a_0) + d_E (a_0, \bar{a}_{N_n}) + d_E (\bar{a}_{N_n}, A_{N_n}) 
    \\&=  d_E (a_{N_n}, a_0) + d_E (a_0, \bar{a}_{N_n}) \to 0,
    \end{align*}
    as $n \to \infty$, which implies 
    \begin{align*}
    \limsup_{n \to \infty} \sup_{a \in A} d_E (a, A_n) &\leq \lim_{m \to \infty} d_E (a_{n_m}, A_{n_m}) = \lim_{m \to \infty} d_E (a_{N_m}, A_{N_m}) = 0. \qedhere 
    \end{align*} 
\end{proof}

\begin{proposition} \label{prop: 2 step to Hausdorff}
Let Assumptions~\ref{SA: control}~\ref{item:SA:continuity in control}--\ref{item:SA:growth} hold.
Moreover, let \(M > 0, (\varepsilon_n)_{n = 1}^\infty \subseteq (0, \infty)\) with \(\varepsilon_n \searrow 0\), and \((t_n, \omega_n)_{n = 1}^\infty \subseteq \bR_+ \times \Omega\) with \((t_n, \omega_n) \to (t, \omega) \in \bR_+ \times \Omega\). Then,
\begin{enumerate}[(a)]
        \item\label{item:prop: 2 step to Hausdorff:a} for every sequence \((P_n)_{n = 1}^\infty\) with \(P_n \in \cR^{\varepsilon_n}_M (t_n, \omega_n)\), the family \(\{ P_n \colon n \in \mathbb{N}\}\) is relatively compact (in the space of probability measures on \((\Omega \times \m^*, \cF \otimes \M^*)\) with the weak topology), and all of its accumulation points \(P\) satisfy \(P \in \cR^0_M (t, \omega)\). 
        \item\label{item:prop: 2 step to Hausdorff:b} the set \(\cR^0_M (t, \omega)\) is compact and, for every \(P \in \cR^0_M (t, \omega)\), there exists a sequence \((P_n)_{n = 1}^\infty\) such that \(P_n \in \cR^{\varepsilon_n}_M (t_n, \omega_n)\) with \(P_n \to P\) and \(P_n (\Omega \times A) = P (\Omega \times A)\) for all \(A \in \mathcal{M}^*\).
    \end{enumerate}
\end{proposition}

\begin{corollary} \label{coro: HD convergence}
Let Assumptions~\ref{SA: control}~\ref{item:SA:continuity in control}--\ref{item:SA:growth} hold.
    For every \(M > 0\), every sequence \((\varepsilon_n)_{n = 1}^\infty\) with \(\varepsilon_n \searrow 0\) and every sequence \((t_n, \omega_n)_{n = 1}^\infty \subseteq \bR_+ \times \Omega\) with \((t_n, \omega_n) \to (t, \omega)\), we have 
    \[
    \cR^{\varepsilon_n}_M (t_n, \omega_n) \to \cR^0_M (t, \omega)  
    \]
    in the Hausdorff metric topology, and \(\cR^0_M (t, \omega)\) is compact (in the space of probability measures on \((\Omega \times \m^*, \cF \otimes \M^*)\) with the weak topology).
\end{corollary}

\begin{proof} This follows directly from Lemma~\ref{lem: HD convergence} and Proposition~\ref{prop: 2 step to Hausdorff}.
\end{proof} 

\begin{proof}[Proof of Proposition~\ref{prop: 2 step to Hausdorff}]
   {\em Proof of \ref{item:prop: 2 step to Hausdorff:a}}: Fix a finite time horizon \(T > \sup_{n \in \mathbb{N}} t_n \vee t\).
   By definition of the set \(\cR^{\varepsilon_n}_M (t_n, \omega_n)\) and the linear growth conditions of \(\mu\) and \(\sigma\) (we denote the linear growth constant by \(C > 0\)), for every \(P \in \cR^{\varepsilon_n}_M (t_n, \omega_n)\), we obtain that 
   \begin{align*}
       E^P \big[ \| \X \|_T^2 \big] 
        &\leq 3\| \omega_n \|^2_T + 3 T E^P \Big[ \int_0^T \int_{\Lambda \times \bR^r} \| \mu (s, \X, \lambda) + \sigma (s, \X, \lambda) z \|^2 \, M^* (ds, d\lambda, dz) \Big] 
        \\& \hspace{3cm}+ 3\varepsilon_n E^P \Big[ \int_0^T \int_{\Lambda \times \bR^r} \| \sigma (s, \X, \lambda) \|^2 \, M^* (ds, d \lambda, dz) \Big]
        \\&\leq 3\| \omega_n \|^2_T + 6T E^P \Big[ \int_0^T \int_{\Lambda \times \bR^r} \big( \| \mu (s, \X, \lambda) \|^2 + \| \sigma (s, \X, \lambda) \|^2 + \| z \|^2 \big) \, M^* (ds, d \lambda, d z) \Big] 
        \\&\hspace{3cm} + 6 C \varepsilon_n E^P \Big[ \int_0^T (1 + \| \X \|_s^2) \, ds \Big]
       \\&
       \leq 3 \| \omega_n \|_T^2 + C (24 T^2 + 6 \varepsilon_n T) + 6 T M + C (24 T + 6\varepsilon_n ) \int_0^T E^P \big[ \| \X \|_s^2 \big] \, ds.
   \end{align*}
   Hence, Gronwall's lemma, modulo a stopping argument that we ignore to ease the presentation, yields that 
   \begin{align} \label{eq: moment bound} 
       E^P \big[ \| \X \|_T^2 \big] \leq C', 
   \end{align}
   where \(C' > 0\) is a constant that only depends on \(C, T, \sup_{n \in \mathbb{N}} \varepsilon_n, M\) and \(\sup_{n \in \mathbb{N}} \| \omega_n \|_T < \infty\).
   Again, the linear growth conditions on \(\mu\) and \(\sigma\) yield that, for every stopping time \(\tau \in [0, T]\) and \(\theta \in (0, T)\), we have
\begin{equation} \label{eq: holder estimate}
\begin{split}
E^P \big[ \| \X_{\tau + \theta} - \X_{\tau} \|^2 \big] &\leq 4 \theta E^P \Big[ \int_\tau^{\tau + \theta} \int_{\Lambda \times \bR^r} \| \mu (u, \X, \lambda) + \sigma (u, \X, \lambda) z \|^2 \, M^* (du, d \lambda, dz ) \Big] \\&\hspace{3cm}+ 2 \varepsilon_n E^P \Big[ \int_\tau^{\tau + \theta} \int_{\Lambda \times \bR^r} \| \sigma (u, \X, \lambda) \|^2 \, M^* (du, d \lambda, dz) \Big] 
\\&\leq \theta^2 C 16 \big( 1 + E^P \big[ \| \X \|^2_{2T} \big] \big) + \theta 4 M + \theta 4 \varepsilon_n C \big( 1 + E^P \big[ \| \X \|^2_{2T} \big] \big)
\\&\leq \theta \text{ const}.
\end{split}
\end{equation}
Consequently, thanks to \eqref{eq: moment bound} and \eqref{eq: holder estimate},  Aldous' tightness criterion (see, e.g., \cite[Theorem~VI.4.5]{JS}) yields that \(\bigcup_{n = 1}^\infty\{ P \circ \X^{-1} \colon P \in \cR^{\varepsilon_n}_M (t_n, \omega_n) \}\) is relatively compact. As the sets
\begin{align} \label{eq: sets in M compact}
\Big\{ M^* \in \mathbb{M}^* \colon \int_0^\infty \int_{\Lambda \times \bR^r} \| z \|^2 \, M^* (ds, d \lambda, dz) \leq \beta \Big\}, \quad \beta > 0,
\end{align} 
are compact in \(\mathbb{M}^*\) by \cite[Theorem~2.5]{balder_01}, a routine application of Chebychev's inequality upgrades the relative compactness to \(\bigcup_{n = 1}^\infty \cR^{\varepsilon_n}_M (t_n, \omega_n)\). This proves the first part of \ref{item:prop: 2 step to Hausdorff:a}.

It remains to prove that for \(P_n \in \cR^{\varepsilon_n}_M (t_n, \omega_n)\) the convergence \(P_n \to P\) implies \(P \in \cR^{0}_M (t, \omega)\). This follows from a standard martingale problem argument; see \cite[Proposition~IX.1.12]{JS}.
Essentially, we need to check the following conditions:
\begin{enumerate}[(i)]
\item For all \(\varphi \in C^2_b (\bR^d; \bR)\), the map \((x, t, \omega, \lambda, h) \mapsto \widetilde{L}^0 (x, t, \omega, \lambda, h, \varphi)\) is continuous; \label{item:a MPA}
\item\label{item:b MPA} For all \(T > 0, x \in \bR^d, s \in [0, T], \omega \in \Omega, \lambda \in \Lambda\), and \(h \in \bR^r\), there exists a constant \(C = C_{T, \varphi} > 0\) such that 
\[
| \widetilde{L}^\varepsilon (x, s, \omega, \lambda, h, \varphi) | \leq C \big( 1 + \| \omega \|_s^2 + \|h\|^2 \big); 
\]
\item\label{item:c MPA} For all \(\varphi \in C^2_b (\bR^d; \bR)\), \(T > 0\), and every compact set \(K \subset \Omega\), it holds that 
\[
\sup \big| \widetilde{L}^\varepsilon (x, s, \omega, \lambda, h, \varphi) - \widetilde{L}^0 (x, s, \omega, \lambda, h, \varphi) \big| \to 0, \quad \text{as } \varepsilon \searrow 0, 
\]
where the sup is taken over all \(x \in \bR^d, s \in [0, T], \omega \in K, \lambda \in \Lambda\) and \(h \in \bR^r\).
\end{enumerate} 
Item \ref{item:a MPA} follows directly from Assumptions~\ref{SA: control}~\ref{item:SA:continuity in control} and \ref{item:SA:Lipschitz}. The items \ref{item:b MPA} and \ref{item:c MPA} are direct consequences of Assumption~\ref{SA: control}~\ref{item:SA:growth}. Now, \ref{item:a MPA} yields that (a condition like) \cite[IX.1.4 (iii)]{JS} holds, \ref{item:c MPA} yields that (a condition like) \cite[IX.1.4 (iv)]{JS} holds, and \ref{item:b MPA} together with the moment bound \eqref{eq: moment bound} and the moment constraints of the controls from the definition of \(\cR_M^{\varepsilon_n} (t_n, \omega_n)\) yield that (a version of) the uniform integrability condition from \cite[IX.1.12]{JS} holds. Thus, a martingale problem argument such as \cite[Proposition~IX.1.12]{JS} yields that \(P \in \cR^0 (t, \omega)\). We omit the details for brevity. 

\smallskip
\noindent
{\em Proof of \ref{item:prop: 2 step to Hausdorff:b}}: Compactness of \(\cR^0_M (t, \omega)\) follows along the lines of the proof of \ref{item:prop: 2 step to Hausdorff:a} above and we omit the details for brevity. For a given \(P \in \cR^0_M (t, \omega)\), we now construct the approximation sequence \((P_n)_{n = 1}^\infty\) with the claimed properties. Take \(P \in \cR^0_M (t, \omega)\) and take the underlying probabilistic setup as a standard extension of \((\Omega \times \m^*, \cF \otimes \M^*, (\cG^*_t)_{t \geq 0}, P)\) that supports orthogonal martingale measures \(N = (N^k)_{k = 1}^r\) with intensity \(M^* (ds, d \lambda, dz)\). To simplify our notation, we still use \(P\) for the underlying probability measure.
Adapting classical methods for SDEs with Lipschitz coefficients, cf.\ \cite[p.~100]{EM_90}, to our underlying stochastic setup, there exists an \(\bR^d\)-valued continuous process \(Y = (Y_t)_{t \geq 0}\) such that 
	\begin{align*}
		\begin{cases} d Y^n_s =  \int_{\Lambda \times \bR^r} ( \mu (s, Y^n, \lambda) + \sigma (s, Y^n, \lambda) \, z ) \, M^* (ds, d \lambda, dz) \\\hspace{3.5cm}+ \sqrt{\varepsilon_n} \, \int_{\Lambda \times \bR^r} \sigma (s, Y^n, \lambda) \, N (ds, d \lambda, dz), & s > t_n, \\
		Y^n_s = \omega_n (s), & s \leq t_n.
        \end{cases}
	\end{align*} 
Now set \(P_n := P \circ (Y^n, M^*)^{-1}\) and notice that this measure is an element of \(\cR^{\varepsilon_n}_M (t_n, \omega_n)\) such that \(P_n (\Omega \times A) = P (\Omega \times A)\) for all \(A \in \M^*\). Thus, it remains to show that \(P_n \to P\), which
is easily seen to be implied by 
\[
\forall \, \xi, T > 0, \quad P ( \| X - Y^n \|_T > \xi ) \to 0,  
\]
as $n \to \infty$. 
Fix \(\xi, T > 0\) and, for \(R > 0\), set 
\[
	\rho_R^n := \inf \Big\{ s \geq 0 \colon \int_{t \vee t^n}^{s + (t \vee t^n)} \int_{\Lambda \times \bR^r} \| z \|^2 \, M^* (dr, d \lambda, dz) \geq R \Big\}. 
\]
Using that \(P \in \cR^0_M (t, \omega)\), we obtain that 
\[
P (\rho^n_R \leq T) \leq \frac{1}{R} \, E^P \Big[ \int_{t \wedge t^n}^T\int_{\Lambda \times \bR^r} \|z\|^2 \, M^* (ds, d\lambda, dz) \Big] \leq \frac{M}{R}.
\]
Hence,
\[
P ( \| X - Y^n \|_T > \xi ) \leq \frac{1}{\xi^2} E^P \big[ \| X - Y^n \|^2_{T \wedge \rho^n_R} \big] + \frac{M}{R}, 
\]
and we observe that it suffices to prove that, for every fixed \(R > 0\), 
\[
E^P \big[ \| X - Y^n \|^2_{T \wedge \rho^n_R} \big] \to 0, 
\]
as $n \to \infty$.
This is the program for the remainder of this proof. 
Recall that \(P\)-a.s. 
\[
\begin{cases} d X_s =\int_{\Lambda \times \bR^r} ( \mu (s, X, \lambda) + \sigma (s, X, \lambda) z ) \, M^* (ds, d\lambda, dz), & s > t, \\ X_s = \omega (s), & s \leq t. \end{cases} 
\]
In the following, \(C > 0\) denotes a generic constant independent of \(n\) that is allowed to change from line to line. 

Using \eqref{eq: moment bound} and the linear growth conditions from Assumption~\ref{SA: control}~\ref{item:SA:growth}, we obtain that
\begin{align*}
    E^P \big[ \| X - Y^n \|^2_{T \wedge \rho^n_R} \big] &\leq C \Big( \| \omega_n ( \, \cdot \wedge t_n ) - \omega (\, \cdot \wedge t ) \|^2_T \\&\qquad + E^P \Big[ \Big\| \int_{t}^{T \wedge \rho^n_R} \int_{\Lambda \times \bR^r} (\mu (s, X, \lambda) + \sigma (s, X, \lambda) z ) \, M^* (ds, d\lambda, dz) 
    \\&\qquad \qquad - \int_{t_n}^{T \wedge \rho^n_R} \int_{\Lambda \times \bR^r} (\mu (s, Y^n, \lambda) + \sigma (s, Y^n, \lambda) z ) \, M^* (ds, d\lambda, dz) \Big\|^2 \Big] + \varepsilon_n \Big). 
\end{align*}
The first term on the right hand side converges to zero by the Arzel\`a--Ascoli theorem and the last term \(\varepsilon_n\) converges to zero by assumption. Using that 
\[
  \| \mu (s, \omega, \lambda) + \sigma (s, \omega, \lambda) z \|^2 \leq C (1 + \|\omega\|_s^2 + \|z\|^2 ), 
\] the moment bounds \eqref{eq: moment bound} and \(E^P [ \int_0^\infty \int_{\Lambda \times \bR^r} \|z\|^2 \, M^* (ds, d\lambda, dz) ] \leq M\), using the Cauchy--Schwarz inequality, we obtain that 
\begin{align*}
    &E^P \Big[  \Big\| \int_{t \wedge t_n}^{t \vee t_n} \int_{\Lambda \times \bR^r} (\mu (s, X, \lambda) - \mu (s, Y^n, \lambda) + (\sigma (s, X, \lambda) - \sigma (s, Y^n, \lambda)) z ) \, M^* (ds, d\lambda, dz) \Big\|^2 \Big] 
    \\&\leq E^P \Big[ \Big( \int_{t \wedge t_n}^{t \vee t_n} \int_{\Lambda \times \bR^r} \|(\mu (s, X, \lambda) - \mu (s, Y^n, \lambda) + (\sigma (s, X, \lambda) - \sigma (s, Y^n, \lambda)) z ) \|\, M^* (ds, d\lambda, dz) \Big)^2 \Big] 
    \\&\leq E^P \Big[  \int_{t \wedge t_n}^{t \vee t_n} \int_{\Lambda \times \bR^r} \|(\mu (s, X, \lambda) - \mu (s, Y^n, \lambda) \\&\hspace{5cm}+ (\sigma (s, X, \lambda) - \sigma (s, Y^n, \lambda)) z )\|^2 \, M^* (ds, d\lambda, dz) \Big] \, |t_n - t|
    \\&
    \leq C | t_n - t | \to 0, 
\end{align*}
as \(n \to \infty\). 
In summary, we obtained the structure
\[
E^P \big[ \| X - Y^n \|^2_{T \wedge \rho^n_R} \big] \leq a_n + C\, I_n, 
\]
where \((a_n)_{n = 1}^\infty \subseteq (0, \infty)\) is a sequence with \(a_n \to 0\), and 
\begin{align*}
I_n := E^P \Big[  \Big\| \int_{t \vee t_n}^{T \wedge \rho^n_R} \int_{\Lambda \times \bR^r} (\mu (s, X, \lambda) &- \mu (s, Y^n, \lambda) \\&+ (\sigma (s, X, \lambda) - \sigma (s, Y^n, \lambda)) z ) \, M^* (ds, d\lambda, dz) \Big\|^2 \Big].
\end{align*}
Using the Lipschitz assumptions from Assumption~\ref{SA: control}~\ref{item:SA:Lipschitz} and H\"older's inequality, we estimate 
\begin{align*}
I_n &\leq C \Big( \int_{t \vee t_n}^T E^P \big[ \| X - Y^n \|^2_{s \wedge \rho^n_R} \big] \, ds \\&\hspace{3cm}+ E^P \Big[ \int_{t \vee t_n}^{T \wedge \rho_R^n} \| X - Y^n \|^2_s \, ds \, \int_{t \vee t_n}^{T \wedge \rho^n_R} \int_{\Lambda \times \bR^r} \| z \|^2 \, M^* (ds, d \lambda, dz) \Big] \Big) 
\\&\leq C (1 + R) \int_{t \vee t_n}^T E^P \big[ \| X - Y^n \|^2_{s \wedge \rho^n_R} \big] \, ds.
\end{align*}
Finally, Gronwall's lemma yields that 
\[
E^P \big[ \| X - Y^n \|^2_{T \wedge \rho^n_R} \big] \leq a_n e^{C T} \to 0,
\]
as $n \to \infty$.
This completes the proof of \(P_n \to P\) and consequently, also Part \ref{item:prop: 2 step to Hausdorff:b}.
\end{proof}

\section{Proof of Theorem~\ref{theo: main result}} \label{sec: proof main}
Take \((t, \omega) \in \T\), \(\varepsilon_n \searrow 0, t_n \to t\) and \(\omega_n \to \omega\).
In the following, we show
\begin{align}
    \limsup_{n \to \infty} V^{\varepsilon_n}_{t_n, \omega_n} &\leq V^{0}_{t, \omega}, \label{eq: to show 1}
    \\ 
    V^0_{t, \omega} &\leq \liminf_{n \to \infty} V^{\varepsilon_n}_{t_n, \omega_n}.
    \label{eq: to show 2}
\end{align}
Together, they entail \eqref{eq: main convergence statement}, proving the claim of Theorem~\ref{theo: main result}. At this point, we stress that the assumption \((t, \omega) \in \T\) is only needed for the first inequality~\eqref{eq: to show 1}.

\smallskip
\noindent
{\em Proof of \eqref{eq: to show 1}}: 
Recall the notation \(D_\delta = \{ x \in \bR^d : \inf_{y \in D} \| x - y \| \le \delta \}\).
Using that \(\tau^{t_n}_D \leq \tau^{t_n}_{D_\delta}\), and applying Theorem~\ref{theo: VF} with \(\varphi = g (\tau^{t_n}_D)\), we get that
\begin{align*}
\varepsilon_n \log \mathcal{E}^{\varepsilon_n}_{t_n} (e^{g (\tau_{D}^{t_n}) / \varepsilon_n}) &\leq
\varepsilon_n \log \mathcal{E}^{\varepsilon_n}_{t_n} (e^{g (\tau_{D_\delta}^{t_n}) / \varepsilon_n}) 
\\&= \sup_{P \in \cR^{\varepsilon_n} (t_n, \omega_n)} E^P \Big[ g (\tau_{D_\delta}^{t_n}) - \frac{1}{2} \int_0^\infty \int_{\Lambda \times \bR^r} \| z \|^2 \, M^* (dt, d\lambda, dz) \Big].
\end{align*}
Moreover, for \(M = 4 \| g \|_\infty\), we have 
\begin{align*}
\sup_{P \in \cR^{\varepsilon_n} (t_n, \omega_n)} &E^P \Big[ g (\tau_{D_\delta}^{t_n}) - \frac{1}{2} \int_0^\infty \int_{\Lambda \times \bR^r} \| z \|^2 \, M^* (dt, d\lambda, dz) \Big] 
\\&= \sup_{P \in \cR^{\varepsilon_n}_M (t_n, \omega_n)} E^P \Big[ g (\tau_{D_\delta}^{t_n}) - \frac{1}{2} \int_0^\infty \int_{\Lambda \times \bR^r} \| z \|^2 \, M^* (ds, d\lambda, dz) \Big], 
\end{align*} 
which follows from the fact that the optimization ignores measures from \(\cR^\varepsilon (t_n, \omega_n) \setminus \cR^\varepsilon_M (t_n, \omega_n)\). 
Using that 
\[
\{ \tau^0_{D_\delta} < s \} = \bigcup_{r \in \mathbb{Q} \, \cap \, [0, s)} \{ \X_r \not \in D_\delta \}  
\]
is open, the map \(\omega' \mapsto \tau_{D_\delta}^0 (\omega')\) is upper semicontinuous. Consequently, the map \((s, \omega') \mapsto \tau^s_{D_\delta} (\omega')= s + \tau^0_{D_\delta} (\theta_s \omega')\), with \(\theta_s \omega' = \omega' (\, \cdot + s )\), is upper semicontinuous as well (following from the fact that \((s, \omega') \mapsto \theta_s \omega'\) is continuous by the Arzel\`a--Ascoli theorem).
As \(m^* \mapsto \int_0^\infty \int_{\Lambda \times \bR^r} \| z \|^2 \, m^* (ds, d\lambda, dz)\) is lower semicontinuous on~\(\mathbb{M}^*\), we conclude that 
\[
(s, P) \mapsto E^P \Big[ g (\tau_{D_\delta}^s) - \frac{1}{2} \int_0^\infty \int_{\Lambda \times \bR^r} \| z \|^2 \, M^* (dr, d\lambda, dz) \Big]
\]
is bounded and upper semicontinuous on \(\bR_+ \times \mathcal{P}_M (\Omega \times \mathbb{M}^*)\), where 
\[
\mathcal{P}_M (\Omega \times \mathbb{M}^*) := \Big\{ P \in \mathcal{P} (\Omega \times \mathbb{M}^*) : E^P \Big[ \int_0^\infty \int_{\Lambda \times \bR^r} \|z\|^2 \, M^* (ds, d\lambda, dz) \Big] \leq M \Big\}. 
\]
Consequently, we deduce from Corollary~\ref{coro: HD convergence} and \cite[Lemma~12.1.7]{SV} that 
\begin{align*}
\limsup_{n \to \infty} \sup_{P \in \cR^{\varepsilon_n}_M (t_n, \omega_n)} E^P \Big[ g (\tau_{D_\delta}^{t_n}) - \frac{1}{2} \int_0^\infty \int_{\Lambda \times \bR^r} \| z \|^2 \, M^* (dt, d\lambda, dz) \Big]
\leq V^{0, \delta}_{t, \omega}.
\end{align*} 
As \((t, \omega) \in \T\), letting \(\delta \searrow 0\) entails \eqref{eq: to show 1}. 

\smallskip
\noindent
{\em Proof of \eqref{eq: to show 2}}: 
Arguing as in the proof of \eqref{eq: to show 1}, with \(M = 4 \|g\|_\infty\), we have 
\[
\varepsilon \log \mathcal{E}^\varepsilon_t (e^{g (\tau_{D}^t) / \varepsilon}) = \sup_{P \in \cR^\varepsilon_M (t, \omega)} E^P \Big[ g (\tau_{D}^t) - \frac{1}{2} \int_0^\infty \int_{\Lambda \times \bR^r} \| z \|^2 \, M^* (dt, d\lambda, dz) \Big], 
\]
As \(D\) is open, the set
\[
\{ \tau^0_D \leq s \} = \Big\{ \inf_{q \in [0, s]} \inf_{y \not \in D} \| \X_q - y \| = 0 \Big\} 
\]
is closed (as \(\omega' \mapsto \inf_{q \in [0, s]} \inf_{y \not \in D} \| \omega' (q) - y \|\) is continuous by Berge's maximum theorem). Thus, the map \(\omega \mapsto \tau^0_D (\omega)\) is lower semicontinuous. As in the previous step, this implies that \((s, \omega') \mapsto \tau^s_D (\omega')\) is lower semicontinuous and the same is true for \((s, P) \mapsto E^P [ g (\tau_{D}^s)]\). 
For an arbitrary \(\xi > 0\), let \(P^0 =P^0_\xi \in \cR^0_M (t, \omega)\) be such that 
\begin{align*}
V^0_{t, \omega} &\leq \sup_{P \in \cR^0_M (t, \omega)} E^P \Big[ g (\tau^{t}_{D}) - \frac{1}{2} \int_0^\infty \int_{\Lambda \times \bR^r} \|z\|^2 \, M^* (ds, d \lambda, dz) \Big] 
\\&\leq E^{P^0} \Big[ g (\tau^{t}_{D}) - \frac{1}{2} \int_0^\infty \int_{\Lambda \times \bR^r} \|z\|^2 \, M^* (ds, d \lambda, dz) \Big] + \xi.
\end{align*} 
Due to Proposition~\ref{prop: 2 step to Hausdorff}~\ref{item:prop: 2 step to Hausdorff:b}, there exists a sequence \((P_n)_{n = 1}^\infty\) with \(P_n \in \cR^{\varepsilon_n}_M (t_n, \omega_n)\), \(P_n \to P^0\), and \(P_n (\Omega \times \cdot \,) = P^0 (\Omega \times \cdot \,)\) on \(\M^*\). Using the lower semicontinuity of \((s, P) \mapsto E^P [g (\tau^s_{D})]\), we get that
\begin{align*}
    E^{P^0} \Big[ g (\tau^{t}_{D}) &- \frac{1}{2} \int_0^\infty \int_{\Lambda \times \bR^r} \|z\|^2 \, M^* (ds, d \lambda, dz) \Big] 
    \\&= E^{P^0} \big[ g (\tau^{t}_{D}) \big] - \lim_{n \to \infty} E^{P_n} \Big[ \frac{1}{2} \int_0^\infty \int_{\Lambda \times \bR^r} \|z\|^2 \, M^* (ds, d \lambda, dz) \Big]
    \\&\leq \liminf_{n \to \infty} E^{P_n} \big[ g (\tau^{t_n}_{D}) \big] - \lim_{n \to \infty} E^{P_n} \Big[ \frac{1}{2} \int_0^\infty \int_{\Lambda \times \bR^r} \|z\|^2 \, M^* (ds, d \lambda, dz) \Big]
    \\&= \liminf_{n \to \infty} E^{P_n} \Big[ g (\tau^{t_n}_{D}) - \frac{1}{2} \int_0^\infty \int_{\Lambda \times \bR^r} \|z\|^2 \, M^* (ds, d \lambda, dz) \Big] 
    \\&\leq \liminf_{n \to \infty} V^{\varepsilon_n}_{t_n, \omega_n}. \phantom \int
\end{align*}
Summing up, we proved that 
\[
V^0_{t, \omega} \leq \liminf_{n \to \infty} V^{\varepsilon_n}_{t_n, \omega_n} + \xi.
\]
As \(\xi > 0\) was arbitrary, \eqref{eq: to show 2} follows. \qed


\appendix

\section{The class of test functions for Definition~\ref{definition:vis_sol}} \label{app: test}
This appendix is dedicated to the definition of the set of test functions used for the path-dependent viscosity concept from Definition~\ref{definition:vis_sol}. We follow the exposition from \cite{cosso_russo_22,CN_23_EJP}.

Let \(T >0\) and consider some initial time \(t_0 \in [0,T)\). Let \(D(\bR_+; \bR^\d)\) be the space of \cadlag functions from \(\bR_+\) into \(\mathbb{R}^\d\) and define \(\Lambda(t_0) := [t_0, T] \times D(\bR_+; \bR^\d)\). On the space $\Lambda(0)$, we consider the pseudometric 
\begin{equation}\label{eq:appendix_d}
    \mathsf{d} ( (t, \omega); (t', \omega') ) := \big( 1 + \| \omega \|_t + \| \omega'\|_{t'} \big) |t - t'|^{\frac{1}{2}} + \| \omega (\cdot \, \wedge t) - \omega' (\cdot \, \wedge t') \|_T
\end{equation}
as above.

We say that a map \(F \colon \Lambda (t_0) \to \bR\) admits a \emph{horizontal derivative} at \((t, \omega) \in \Lambda (t_0)\) with $t<T$ if
\[
\p F (t, \omega) := \lim_{h \searrow 0} \frac{ F(t + h, \omega(\cdot \wedge t)) - F(t, \omega(\cdot \wedge t))}{h}
\]
exists.
At \(t = T\), the horizontal derivative is defined as its left limit, i.e., we have
\[
\p F(T, \omega) := \lim_{h \nearrow T} \p F(h,\omega).
\]
Moreover, we say that 
\(F\) admits a \emph{vertical derivative} at \((t, \omega) \in \Lambda (t_0)\) with $t < T$ if 
\[
\partial_i F (t, \omega) := \lim_{h \to 0} \frac{F (t, \omega + h e_i\1_{[t,T]}) - F (t, \omega)}{h}, \quad i = 1, 2, \dots, \d,
\]
exists, where \(e_1, \dots, e_\d\) is an orthonormal basis of \(\bR^\d\).
Accordingly, the second vertical derivatives \(\partial^2_{ij} F(t,\omega), i, j = 1, \dots, \d,\) at \((t, \omega) \in \Lambda (t_0) \) are defined as
\[
\partial^2_{ij} F(t,\omega) := \partial_i (\partial_j F)(t, \omega).
\]
We write \(\nabla F := (\partial_1 F, \dots, \partial_\d F)\) for the {\em vertical gradient} and \(\nabla^2 F := (\partial_{ij}^2 F)_{i, j = 1, \dots, \d}\) for the {\em vertical Hessian matrix}.

We denote by \(C^{1,2}(\Lambda (t_0); \bR)\) the set of functions \(F \colon \Lambda (t_0) \to \bR\) that are continuous with respect to \(\mathsf{d}\), 
such that
\(\p F, \nabla F\), and \(\nabla^2 F\)
exist everywhere on $\Lambda (t_0)$ and are continuous with respect to \(\mathsf{d}\).

Moreover, for $\of t_0, T \gs \coloneqq [t_0, T] \times \Omega$, the set 
\(C^{1,2}( \of t_0, T \gs ; \bR)\) 
contains all function \(F \colon \of t_0, T \gs  \to \bR \) 
such that there exists a function \(\hat{F} \in C^{1,2}(\Lambda(t_0); \bR)\) such that, for all $(t, \omega) \in \of t_0, T \gs$, we have $F(t, \omega) = \hat{F}(t,\omega)$.
For any such $F \in C^{1,2}( \of t_0, T \gs ; \bR)$, we define
\[
\p F(t, \omega) := \p \hat{F}(t, \omega), \quad
\nabla F(t,\omega) := \nabla \hat{F}(t,\omega), \quad
\nabla^2 F(t,\omega) := \nabla^2 \hat{F}(t,\omega)
\]
for \((t, \omega) \in \of t_0, T \gs \).
By \cite[Lemma 2.4]{Z_23}, the derivatives \(\p F, \nabla F, \nabla^2 F \) are well-defined for \(F \in C^{1,2}(\of t_0, T \gs ; \bR)\).
Finally, let \(C_{pol}^{1,2}( \of t_0, T \gs ; \bR) \) be the set of all functions \(F \in C^{1,2}( \of t_0, T \gs ; \bR)\) such that there exist constants \(C, q \geq 0\) with
\[
| \p F(t, \omega) | + \| \nabla F(t,\omega) \| + \| \nabla^2 F(t,\omega) \| \leq C \Big( 1 + \| \omega( \cdot \wedge t ) \|^q\Big)
\]
and \((t, \omega) \in \of t_0, T \gs \).

\printbibliography

\end{document}